\documentclass[12pt,a4paper]{amsart}
\usepackage{amsmath, amssymb, amsthm, yhmath}
\input cyracc.def

\usepackage{graphicx}
\usepackage{caption} 
\usepackage{float} 
\usepackage{mathtools} 
\usepackage[colorlinks=false,
  linkbordercolor={0 0 1}, 
  citebordercolor={0 1 0}, 
  urlbordercolor={1 0 0}  
]{hyperref}
\numberwithin{equation}{section}
\usepackage{enumerate}
\newtheorem{theorem}{Theorem}[section]
\newtheorem{lemma}[theorem]{Lemma}

\newtheorem{corollary}[theorem]{Corollary}
\newtheorem{proposition}[theorem]{Proposition}

\theoremstyle{definition} 

\newtheoremstyle{remarkstyle} 
 {} 
 {} 
 {\normalfont} 
 {} 
 {\itshape} 
 {.}  
 { }  
 {} 
\theoremstyle{remarkstyle}
\newtheorem*{remark}{Remark} 

\allowdisplaybreaks[4]

\begin{document}

\author[K. Matsuzaki]{Katsuhiko Matsuzaki \textsuperscript{$a$}} 
\address{$a.$ Department of Mathematics, School of Education, Waseda University, Tokyo 169-8050, Japan} 
\email{matsuzak@waseda.jp} 

\author[F. Tao]{Fei Tao \textsuperscript{$b$}} 
\address{$b.$ Beijing International Center for Mathematical Research, Peking University, Beijing 100871, P. R. China} 
\email{ferrytau@pku.edu.cn}

\thanks{The first author is partially supported by Japan Society for the Promotion of Science (KAKENHI 23K25775 and 23K17656); the second author is partially supported by National Key R \& D Program of China (2025YFA1017500)}


\title[Little Bloch and $\mathrm{VMOA}$]
{Boundary Characterizations of Little Bloch and \\
$\mathrm{VMOA}$ Functions on the Half-Plane}

\begin{abstract}
We extend Pommerenke's characterizations of boundary curves of conformal mappings in terms of the little Bloch and $\mathrm{VMOA}$ conditions from the unit disk $\mathbb{D}$ to the upper half-plane $\mathbb{H}$. 
Let $G\colon \mathbb{H}\to\Omega$ be a conformal mapping onto an unbounded quasidisk $\Omega$ with $G(\infty)=\infty$, and let $g\colon \mathbb{R}\to\Gamma=\partial\Omega$ be its boundary extension. 
In the non-compact setting, the Euclidean smallness on the boundary curve is not necessarily comparable to the smallness of the parameter on $\mathbb{R}$. 
To overcome this difficulty, we use relative versions of the asymptotic conformality and the asymptotic smoothness with respect to the parametrization $g$. 
We prove that $\log G'\in B_0(\mathbb{H})$ is equivalent to the asymptotic conformality of $\Gamma$ relative to $g$, and also to the asymptotic symmetry of the embedding $g$. 
We further prove that $\log G'\in \mathrm{VMOA}(\mathbb{H})$ is equivalent to the asymptotic smoothness of $\Gamma$ relative to $g$, and also to the asymptotic smoothness of $g$. 
These results provide half-plane analogues of Pommerenke's theorems and clarify the role of the parametrization in the unbounded case.
\end{abstract}

\maketitle

\section{Introduction}\label{sec:introduction}

Let $\mathbb{H}=\{z\in\mathbb{C}:\operatorname{Im}z>0\}$ be the upper half-plane. 
For any conformal mapping $G$ of $\mathbb{H}$ into $\mathbb{C}$, $\log G'$ is a Bloch function in the sense that $\sup_{z\in\mathbb{H}}\, (\operatorname{Im}z)\left|(\log G')'(z)\right|<\infty$. 
We say that the Bloch function $\log G'$ belongs to the little Bloch space $B_0(\mathbb{H})$ if
\[
\lim_{\operatorname{Im}z\to 0_+}\,
(\operatorname{Im}z)\left|(\log G')'(z)\right|=0
\]
uniformly in $\operatorname{Re}z\in\mathbb{R}$. 
We also consider the analytic space $\mathrm{VMOA}(\mathbb{H})$, characterized either by the vanishing mean oscillation of its boundary values on $\mathbb{R}$, or equivalently by the vanishing Carleson measure condition for
\[
(\operatorname{Im}z)|(\log G')'(z)|^2\, dx\,dy.
\]
The purpose of this paper is to relate these analytic conditions to the small-scale geometry of the boundary curve.

Our work is motivated by the classical theorems of Pommerenke \cite{Po78}. Let $F$ be a conformal mapping of the unit disk $\mathbb{D}$ onto a bounded Jordan domain. 
Pommerenke showed that the boundary curve $\partial F(\mathbb{D})$ is an asymptotically conformal quasicircle if and only if $\log F'$ belongs to the little Bloch space on $\mathbb{D}$. 
He also proved the sharper, rectifiable version: $\partial F(\mathbb{D})$ is asymptotically smooth if and only if $\log F'$ belongs to
$\mathrm{VMOA}(\mathbb{D})$. 
These results establish a precise correspondence between the asymptotic geometry of the boundary curve and the vanishing analytic behavior of $\log F'$.

This paper extends these theorems to the non-compact half-plane setting. 
Throughout the paper, unless otherwise stated, we consider a conformal mapping $G\colon \mathbb{H}\to\Omega$ onto an unbounded quasidisk $\Omega\subset\mathbb{C}$ such that $G(\infty)=\lim_{z\to\infty}\,G(z)=\infty
$, and we let $g\colon\mathbb{R}\to\Gamma=\partial\Omega$ be its homeomorphic boundary extension. 
Since $\Omega$ is a quasidisk, the map $G$ extends to a quasiconformal self-homeomorphism of the complex plane, and therefore $g$ is a quasisymmetric embedding of $\mathbb{R}$ onto the quasicircle $\Gamma$.

A basic difficulty in the half-plane setting is the loss of compactness. 
In the unit disk, the boundary is compact, and hence the boundary parametrization and its inverse are automatically uniformly continuous. 
In the half-plane setting, however, intervals on $\mathbb{R}$ may drift to infinity, and neither $g$ nor $g^{-1}$ is necessarily uniformly continuous. 
Consequently, the usual definition of asymptotic conformality based solely on the Euclidean chord length is no longer adequate. 
To overcome this, we introduce a relative version of asymptotic conformality. 

Let $\gamma\colon X\to\Gamma$ be a parametrization of a quasicircle by a metric space $(X,d_X)$. 
We say that $\Gamma$ is \emph{asymptotically conformal relative to
$\gamma$} if
\[
\lim_{t\to 0}\,
\sup_{z_1,z_2\in \Gamma}\,\left\{
\max_{z\in\wideparen{z_1z_2}}
\frac{|z_1-z|+|z-z_2|}{|z_1-z_2|}
:
0<d_X(\gamma^{-1}(z_1),\gamma^{-1}(z_2))\leq t
\right\}
=1.
\]
Here $\wideparen{z_1z_2}$ denotes the subarc of $\Gamma$ joining $z_1$ and $z_2$. 
When $\gamma$ is the identity parametrization of $\Gamma$, this reduces to the usual notion of asymptotic conformality. 
In the present paper, the relevant parametrization is the boundary correspondence $g\colon \mathbb{R}\to\Gamma$.

The first main result of the paper is the following half-plane version of
Pommerenke's little Bloch theorem.

\begin{theorem}\label{thm:intro-little-bloch}
Let $G\colon \mathbb{H}\to\Omega$ be a conformal mapping onto an unbounded quasidisk with $G(\infty)=\infty$, and let $g\colon \mathbb{R}\to\Gamma=\partial\Omega$ be its boundary extension. 
Then the following conditions are equivalent:
 \begin{enumerate}
 \item $\log G'\in B_0(\mathbb{H})$;
 \item $\Gamma$ is an asymptotically conformal quasicircle relative to
 $g\colon \mathbb{R}\to\Gamma$;
 \item $g\colon \mathbb{R}\to\mathbb{C}$ is an asymptotically symmetric embedding with image $\Gamma$.
 \end{enumerate}
\end{theorem}

Here, a quasisymmetric embedding $g\colon \mathbb{R}\to\mathbb{C}$ is said to be \emph{asymptotically symmetric} if, for every $\varepsilon>0$ and every $t>0$, there exists $\delta>0$ such that
\[
\frac{|x-a|}{|x-b|}\leq t
\quad\implies\quad
\frac{|g(x)-g(a)|}{|g(x)-g(b)|}\leq (1+\varepsilon)t 
\]
whenever $x,a,b\in\mathbb{R}$ lie in an interval of length less than $\delta$.

Thus Theorem~\ref{thm:intro-little-bloch} establishes that the vanishing analytic behavior of the pre-Schwarzian derivative $(\log G')'$, the relative asymptotic conformality of the boundary curve, and the infinitesimal symmetry of the boundary parametrization are all equivalent.

The proof of Theorem~\ref{thm:intro-little-bloch} proceeds in several steps. We first prove one implication under the additional assumption that $g$ is uniformly continuous. 
In this case, the condition that $\operatorname{dist}(G(x+iy),\Gamma)\to0$ as $y\to0_+$ can be made uniform in $x$, which reduces the problem to Pommerenke's theorem for the disk via a family of uniformly normalized stadium domains over intervals of a fixed length. 
We then prove the converse implication under the corresponding assumption that $g^{-1}$ is uniformly continuous. 
These preliminary statements clarify
exactly where non-compactness enters the argument.

The uniform continuity assumptions are then removed by using the relative
notion introduced above. 
The implication from the relative asymptotic conformality to $\log G'\in B_0(\mathbb{H})$ follows by adapting Pommerenke's compactness argument to a sequence of normalized domains; the relative condition provides the required small-scale control even when the relevant intervals drift to infinity.
The converse implication is obtained more directly by estimating $G'$ in
rectangles near the boundary. 

Integrating $G'$ along polygonal paths then gives the asymptotic symmetry of the boundary map $g$. 
Combining this with the metric characterization of asymptotically symmetric embeddings yields Theorem~\ref{thm:intro-little-bloch}.

We next consider Pommerenke's $\mathrm{VMOA}$ theorem. 
Recall that a locally rectifiable Jordan curve $\Gamma$ is \emph{chord-arc} if the length of the subarc $\wideparen{ab}$ (when $\Gamma$ is a closed curve, the subarc of smaller diameter joining $a$ and $b$) is bounded by a constant multiple of $|a-b|$. 
It is \emph{asymptotically smooth} if the chord-arc constant tends to $1$ uniformly at small scales:
\[
\lim_{|a-b|\to0}\,\frac{\ell(\wideparen{ab})}{|a-b|}=1,
\]
where $\ell(\wideparen{ab})$ denotes the length of the arc $\wideparen{ab}$.
As in the asymptotically conformal case, we need a relative formulation. 

Let $\gamma\colon X\to\Gamma$ be a parametrization of a chord-arc curve. 
We say that $\Gamma$ is \emph{asymptotically smooth relative to $\gamma$} if
\[
\lim_{t\to0}\,\sup_{z_1,z_2\in \Gamma}\,\left\{\frac{\ell(\wideparen{z_1z_2})}{|z_1-z_2|}:0<d_X(\gamma^{-1}(z_1),\gamma^{-1}(z_2))\leq t\right\}=1.
\]
For the boundary parametrization $g\colon \mathbb{R}\to\Gamma$, this condition means that the chord-arc constant of $g([a,b])$ tends to $1$ uniformly in $a$ and $b$ as $|a-b|\to0$, even without assuming that the Euclidean smallness on $\Gamma$ is uniformly equivalent to the smallness of the parameter on $\mathbb{R}$.

The second main result is the following half-plane version of Pommerenke's $\mathrm{VMOA}$ theorem.

\begin{theorem}\label{thm:intro-vmoa}
Let $G\colon \mathbb{H}\to\Omega$ be a conformal mapping onto an unbounded quasidisk with $G(\infty)=\infty$, and let $g\colon \mathbb{R}\to\Gamma=\partial\Omega$ be its boundary extension. 
Then the following conditions are equivalent:
 \begin{enumerate}
 \item $\log G'\in \mathrm{VMOA}(\mathbb{H})$;
 \item $\Gamma$ is an asymptotically smooth chord-arc curve relative to
 $g\colon \mathbb{R}\to\Gamma$;
 \item $g\colon \mathbb{R}\to\mathbb{C}$ is an asymptotically smooth embedding whose image $\Gamma$ is a chord-arc curve.
\end{enumerate}
\end{theorem}
Here, the asymptotic smoothness of the embedding $g\colon \mathbb{R}\to\mathbb{C}$ onto a chord-arc curve $\Gamma$ can be defined as follows:
for every $\varepsilon>0$ and $t>0$ there exists $\delta>0$ such that
\[
\frac{|x-a|}{|x-b|}\leq t
\quad\implies\quad
\frac{\ell(\wideparen{g(x)g(a)})}{|g(x)-g(b)|}\leq (1+\varepsilon)t
\]
whenever $x,a,b\in\mathbb{R}$ lie in an interval with $|x-a| < \delta$ and $|x-b| < \delta$ (see Section~\ref{sec:preliminaries} for details).

The proof of Theorem~\ref{thm:intro-vmoa} combines real-variable estimates with the conformal-geometric arguments used for the little Bloch theorem. 
If $\log G'\in\mathrm{VMOA}(\mathbb{H})$, then its boundary values belong to
$\mathrm{VMO}(\mathbb{R})$. 
Together with the John--Nirenberg inequality, this yields uniform estimates of the ratio between the arc length of $g([a,b])$ and the chord length $|g(a)-g(b)|$ on sufficiently small intervals $[a,b]$. 
It then follows that $g$ is asymptotically smooth.

Conversely, suppose that the boundary curve is asymptotically smooth in the
relative sense. 
The chord-arc condition first implies the $\mathrm{BMOA}$ regularity of $\log G'$. 
One then localizes the problem over intervals of $\mathbb{R}$ by means of square domains and normalized conformal maps from the unit disk. 
On each localized disk, Pommerenke's argument converts the asymptotic smoothness into the vanishing mean oscillation of the logarithmic boundary derivative. 
The relative formulation ensures that this vanishing is uniform over all intervals, including those escaping to infinity. 
Finally, the invariance of $\mathrm{VMO}$ under the Hilbert transform allows one to pass from the real part to the full boundary values of $\log G'$, and hence to conclude that
$\log G'\in\mathrm{VMOA}(\mathbb{H})$.

Together, the two main theorems in this paper show that the corresponding results for the unit disk remain valid in the half-plane setting once the appropriate relative geometric notions are introduced. 
We note that the relative formulation for curves is equivalent to the original one for the unit disk, so our results may be regarded as generalizations of Pommerenke's.

We also record a dimensional consequence of the little Bloch theorem. 
Although quasicircles may have Hausdorff dimension arbitrarily close to $2$ \cite[Theorem 1]{Smi10}, the asymptotically conformal condition forces the Hausdorff dimension to be exactly $1$. 
We prove this first for bounded asymptotically conformal curves by showing that the corresponding conformal map is locally H\"older continuous for every exponent less than $1$. 
The unbounded case then follows by localization, using the half-plane results developed earlier in the paper.

We now conclude the introduction by outlining the organization of the paper. 
Section~\ref{sec:preliminaries} collects the necessary preliminaries. 
We review Bloch and little Bloch functions on $\mathbb{H}$, along with $\mathrm{BMOA}$ and $\mathrm{VMOA}$ spaces. 
We then recall the definitions of quasisymmetric, asymptotically symmetric, and asymptotically smooth embeddings, as well as those of chord-arc curves and asymptotically conformal quasicircles. 
Additionally, we provide a simple half-plane version of the standard Bloch estimate for the logarithmic derivative of a conformal map.

In Section~\ref{sec:asymptotic_conformality_implies_little_bloch} we prove that the asymptotic conformality of the boundary implies the little Bloch condition. 
The first result is obtained under the auxiliary assumption that the boundary map $g$ is uniformly continuous. 
The proof uses stadium domains over intervals of a fixed length and reduces the argument locally to Pommerenke's disk theorem. 
We also prove a local version that does not require uniform continuity.

Section~\ref{sec:little_bloch_implies_asymptotic_conformality} concerns the converse implication under the dual assumption that $g^{-1}$ is uniformly continuous. 
This section follows Pommerenke's original method based on the Visser--Ostrowski quotient and shows how the little Bloch condition yields the asymptotic conformality of the boundary curve.

Section~\ref{sec:asymptotically_symmetric_embeddings} removes the uniform continuity assumptions by using the notion of the relative asymptotic conformality. 
We compare this geometric condition with the asymptotic symmetry of embeddings, extend a theorem of Brania and Yang \cite{BY} to the present non-compact setting, and prove Theorem~\ref{thm:intro-little-bloch}, thereby completing the half-plane analogue of Pommerenke's little Bloch characterization. 

Section~\ref{sec:hausdorff_dimension_of_asymptotically_conformal_curves} is devoted to the Hausdorff dimension of these curves. 
We prove that bounded asymptotically conformal quasicircles have Hausdorff dimension $1$, and then extend the same conclusion to the unbounded cases via localization.

Finally, Section~\ref{sec:asymptotic_smoothness_and_vmoa} treats the $\mathrm{VMOA}$ cases. 
After recalling the John--Nirenberg
inequality in a local form, we prove that $\log G'\in\mathrm{VMOA}(\mathbb{H})$ implies the asymptotic smoothness of the boundary embedding. 
We then prove the converse implication from the relative asymptotic smoothness to $\mathrm{VMOA}$, using localized square domains and Pommerenke's argument on the unit disk. 
These results are assembled into Theorem~\ref{thm:intro-vmoa}, the half-plane analogue of Pommerenke's $\mathrm{VMOA}$ characterization.

\medskip
\noindent
{\bf Acknowledgements.}
The authors would like to thank Huaying Wei for her helping explanations concerning Pommerenke's work.

\section{Preliminaries}\label{sec:preliminaries}

\subsection{Bloch functions}
A holomorphic function $F$ on the upper half-plane $\mathbb{H}$ is called a {\it Bloch function} if it satisfies 
\[
\sup_{z \in \mathbb{H}}\,(\operatorname{Im}z)|F'(z)| <\infty.
\]
The Banach space of all Bloch functions (modulo additive constants) is denoted by $B(\mathbb{H})$. If $F \in B(\mathbb{H})$ vanishes at the boundary in the sense that 
\[
\lim_{\operatorname{Im}z \to 0_+}\,(\operatorname{Im}z)|F'(z)| =0
\]
uniformly in $\operatorname{Re}z \in \mathbb{R}$, then $F$ is called a {\it little Bloch function}. 
The collection of all such functions forms a closed subspace of $B(\mathbb{H})$, denoted by $B_0(\mathbb{H})$. 
More generally, $F$ is called a {\it locally little Bloch function} if for every bounded interval $I \subset \mathbb{R}$,
\[
\lim_{\operatorname{Im}z \to 0_+}\,(\operatorname{Im}z )|F'(z)| =0
\]
uniformly for $\operatorname{Re}z \in I$. 
The set of all such functions is denoted by $B_0^{\mathrm{loc}}(\mathbb{H})$. 

It is a classical result that for a conformal mapping $F$ on the unit disk $\mathbb{D}$, the Bloch norm of $\log F'$ defined by the supremum of $(1-|z|^2)|(\log F')'(z)|$ is bounded by $6$ (see \cite[Proposition 4.1]{Pom}). 
For a conformal mapping $G$ on the upper half-plane $\mathbb{H}$, the corresponding bound is the same, and it becomes $3$ if we use the density $\operatorname{Im}w$ instead of $2\operatorname{Im}w$.
For the convenience of the reader, we include a simple proof. See also \cite[Theorem 5.3.1]{Su}.

\begin{proposition}\label{bloch_norm_H:label}
Let $G$ be a conformal mapping of the upper half-plane $\mathbb{H}$ into $\mathbb{C}$. 
Then
\[
\sup_{w \in \mathbb{H}} \,(\operatorname{Im}w)\left| (\log G')'(w) \right| \leq 3.
\]
\end{proposition}

\begin{proof}
Let $\tau(z) = i\frac{1-z}{1+z}$ be the M\"obius transformation from $\mathbb{D}$ onto $\mathbb{H}$, and let $F = G \circ \tau$.
A direct computation gives
\[
(\operatorname{Im} w)\left|\frac{G''(w)}{G'(w)}\right|= \frac{1}{2}\left|(1-|z|^2)\frac{F''(z)}{F'(z)}+ \frac{2(1-|z|^2)}{1+z}\right|,
\]
where $w = \tau(z)$.

By a classical estimate for conformal mappings on the unit disk $\mathbb{D}$ (see \cite[Proposition 1.2]{Pom}), we have 
\[
\left|(1-|z|^2)\frac{F''(z)}{F'(z)} - 2\bar z\right| \leq 4.
\]
It follows that
 \begin{align*}
 (\operatorname{Im} w)\left|\frac{G''(w)}{G'(w)}\right|
 &= \frac{1}{2}\left|2\bar z + \frac{2(1-|z|^2)}{1+z} + \left((1-|z|^2)\frac{F''(z)}{F'(z)} - 2\bar z\right) \right|\\
 &= \frac{1}{2}\left|2\cdot \frac{1+\bar z}{1+z} + \left((1-|z|^2)\frac{F''(z)}{F'(z)} - 2\bar z\right) \right|\leq 3.
 \end{align*}
This completes the proof.
\end{proof}

\subsection{VMOA functions}
Let $I_0$ be a closed interval on $\mathbb{R}$. 
A locally integrable function $u \in L^1_{\mathrm{loc}}(I_0)$ is said to be of \emph{bounded mean oscillation} if
\[
\|u\|_{\mathrm{BMO}(I_0)}\coloneqq\sup_{I \subset I_0}\frac{1}{|I|}\,\int_I |u(x) - u_I|\,dx< \infty,
\]
where the supremum is taken over all bounded intervals $I \subset I_0$, 
and $u_I$ denotes the integral mean of $u$ over $I$. 
The collection of such functions is denoted by $\mathrm{BMO}(I_0)$, which becomes a Banach space after identifying functions that differ by a constant. 
A function $u \in \mathrm{BMO}(\mathbb{R})$ is said to be of \emph{vanishing mean oscillation} if
\[
\lim_{|I| \to 0}\,\frac{1}{|I|}\int_I |u(x) - u_I|\,dx= 0.
\]
The set of all such functions is denoted by $\mathrm{VMO}(\mathbb{R})$, which is a closed subspace of $\mathrm{BMO}(\mathbb{R})$ introduced in \cite{Sar75}. 

A holomorphic function $F$ on $\mathbb{H}$ belongs to $\mathrm{BMOA}(\mathbb{H})$ if its boundary values, taken in the sense of non-tangential limits, lie in $\mathrm{BMO}(\mathbb{R})$. 
It belongs to $\mathrm{VMOA}(\mathbb{H})$ if, in addition, its boundary values lie in $\mathrm{VMO}(\mathbb{R})$. 
Equivalently, $F \in \mathrm{BMOA}(\mathbb{H})$ (or $\mathrm{VMOA}(\mathbb{H})$) if and only if the measure $|F'(z)|^{2}\, y\, dx\, dy$ is a Carleson measure (respectively, a vanishing Carleson measure) on $\mathbb{H}$ (see, for example, \cite{Pom,Gi}).

\subsection{Asymptotically symmetric and smooth embeddings}

A \emph{quasisymmetric embedding} is an embedding $F\colon X \to Y$ between metric spaces such that there exists an increasing homeomorphism $\eta\colon [0,\infty) \to [0,\infty)$ satisfying
\begin{equation}\label{eq:qs}\tag{QS}
\frac{|x-a|}{|x-b|} \leq t \implies \frac{|F(x)-F(a)|}{|F(x)-F(b)|} \leq \eta(t),
\end{equation}
for all distinct points $x,a,b \in X$. 
It follows directly from \eqref{eq:qs} that the inverse of a quasisymmetric embedding is quasisymmetric, and that the composition of quasisymmetric embeddings is again quasisymmetric.

Quasisymmetry was introduced by Beurling and Ahlfors \cite{BA} in their study of the boundary correspondence of quasiconformal self-maps of the upper half-plane. 
Later, Tukia and V\"ais\"al\"a \cite{TV} formulated a general definition of quasisymmetry in the setting of metric spaces.

In particular, if $F\colon \mathbb{R}\to\mathbb{C}$ is a quasisymmetric embedding, then $F(\mathbb{R})$ is a quasicircle (see \cite{Ah}). Equivalently, a quasicircle can be realized as the image of the real axis under a quasiconformal self-map of the plane. 

A quasisymmetric embedding $F\colon X \to Y$ between metric spaces is called \emph{asymptotically symmetric} if for all $\varepsilon > 0$ and $t > 0$ there exists a $\delta > 0$ such that 
\begin{equation}\label{eq:AS}\tag{AS}
\frac{|x - a|}{|x - b|} \leq t \implies \frac{|F(x) - F(a)|}{|F(x) - F(b)|} \leq (1 + \varepsilon) t, 
\end{equation}
where $x, a, b \in X$ are distinct points contained in a ball of radius $\delta$.

The notion of asymptotically symmetric embeddings was first introduced by Brania and Yang \cite{BY}. 
In their paper, an asymptotically symmetric embedding is not assumed to be quasisymmetric. 
In contrast, throughout this paper we require an asymptotically symmetric embeddings to be quasisymmetric. In the compact case, quasisymmetry is automatic.

\subsection{Asymptotically conformal and smooth quasicircles}

We formally recall some definitions from Section~\ref{sec:introduction}. Let $\Gamma$ be a Jordan curve passing through $\infty$. 
A quasicircle $\Gamma$ is said to be \emph{asymptotically conformal} or \emph{asymptotically symmetric} if
\begin{equation}\label{AC}
\lim_{|a-b| \to 0}\, \max_{w \in \wideparen{ab}} \frac{|a-w|+|w-b|}{|a-b|}=1
\end{equation}
for any $a,b \in \Gamma$ with $|a-b|\to 0$ uniformly. 
A locally rectifiable curve $\Gamma$ is called {\em chord-arc} if there exists $C \geq 1$ such that $\ell(\wideparen{ab})/|a-b| \leq C$. 
The definition of a chord-arc curve remains equivalent if the Euclidean distance $|a - b|$ is replaced by the spherical distance $d_{\widehat {\mathbb{C}}}(a,b)$. 
In particular, the chord-arc property for a Jordan curve $\Gamma \subset \widehat {\mathbb{C}}$ is invariant under M\"obius transformations (see \cite[p.~877]{Mac} and \cite[Section~7]{Tu}). 
A chord-arc curve $\Gamma$ is said to be \emph{asymptotically smooth} if the ratio of arc length to chord length tends to $1$ uniformly at small scales. More precisely, the curve $\Gamma$ satisfies
\[
\lim_{|a-b| \to 0} \,\frac{\ell(\wideparen{a b})}{|a - b|} = 1
\]
uniformly. 
The above definition of asymptotic smoothness is the natural extension to curves passing through $\infty$ of the corresponding notion for bounded Jordan curves introduced in \cite{Po78}.

We say that an asymptotically symmetric embedding $g$ parametrizing a chord-arc curve is \emph{asymptotically smooth} if, for every $0<\varepsilon<1$, there exists $\delta>0$ such that for all $a, b \in \mathbb{R}$,
\begin{align}\label{D_1}\tag{$\rm D_1$}
 0<|a-b|<\delta\;\implies\;\frac{\ell(\wideparen{g(a)g(b)})}{|g(a)-g(b)|}\leq 1+\varepsilon.
\end{align}
Clearly, this condition is equivalent to saying that $\Gamma$ is asymptotically smooth relative to $g$, which has been defined in the introduction.

To provide an analogue of \eqref{eq:AS}, we can equivalently characterize an asymptotically smooth embedding as follows: for every $0<\varepsilon<1$ and $t>0$, there exists $\delta_1>0$ such that for all distinct $x, a, b \in \mathbb{R}$ with $|x-a| < \delta_1$ and $|x-b| < \delta_1$,
\begin{align}\label{D_2}\tag{$\rm D_2$}
 \frac{|x-a|}{|x-b|}\leq t\;\implies\;\frac{\ell(\wideparen{g(x)g(a)})}{|g(x)-g(b)|}\leq (1+\varepsilon)t.
\end{align}
\begin{remark}
 The implication \eqref{D_2} $\Rightarrow$ \eqref{D_1} follows immediately by letting $b\to a$ in \eqref{D_2}. 
 Conversely, assume that \eqref{D_1} holds. 
 Given $0<\varepsilon<1$ and $t>0$, choose $\varepsilon_0=\varepsilon/3$. 
 By our assumptions, there exists $\delta_0>0$ such that \eqref{eq:AS} holds for $\varepsilon_0$, and there exists $\delta>0$ such that \eqref{D_1} holds for $\varepsilon_0$. 
 Let $\delta_1 = \min\{\delta_0, \delta\}$. If $|x-a|,|x-b|<\delta_1$, then 
 \begin{align*}
 \frac{|x-a|}{|x-b|}\leq t\;\implies\;\frac{\ell(\wideparen{g(x)g(a)})}{|g(x)-g(b)|}\leq(1+\varepsilon_0)\frac{|g(x)-g(a)|}{|g(x)-g(b)|}\leq (1+\varepsilon_0)^2t\leq (1+\varepsilon)t.
\end{align*}
\end{remark}

\section{Asymptotic conformality implies little Bloch}\label{sec:asymptotic_conformality_implies_little_bloch}

It was proved in \cite[Theorem 1]{Po78} that for a conformal mapping $F$ from the unit disk $\mathbb{D}$ onto a bounded Jordan domain, the boundary curve $\partial F(\mathbb{D})$ is an asymptotically conformal quasicircle if and only if $\log F'$ belongs to the little Bloch space on $\mathbb{D}$. 
Theorems~\ref{AS->B_0} and~\ref{B_0->AS} below provide the half-plane analogues of this result. 
We extend them to the non-compact setting by assuming that the boundary maps $g$ and $g^{-1}$ are uniformly continuous, respectively.

\begin{theorem}\label{AS->B_0}
Under the additional assumption that $g$ is uniformly continuous, if $\Gamma$ is an asymptotically conformal quasicircle, then $\log G'$ belongs to $B_0(\mathbb{H})$, the little Bloch space on the upper half-plane $\mathbb{H}$.
\end{theorem}

For the proof of this theorem, as well as for later arguments, we need equivalent characterizations for the uniform continuity of $g$.

\begin{lemma}\label{uc1}
The following conditions are equivalent:
 \begin{enumerate}
 \item
 $\lim_{y \to 0}\, d(G(x+iy),\Gamma)=0$ uniformly for $x\in\mathbb{R}$;
 \item
 $d(G(x+iy),\Gamma)$ is uniformly bounded for all $x\in \mathbb{R}$ and for some $y>0$;
 \item
 $g$ is uniformly continuous.
 \end{enumerate}
\end{lemma}

\begin{proof}
The implication $(1) \Rightarrow (2)$ is obvious. We shall prove that
$(2) \Rightarrow (3)$ and $(3) \Rightarrow (1)$.

Since $G$ extends to a quasiconformal self-homeomorphism of $\mathbb{C}$, it is quasisymmetric (see \cite[Theorem 3.5.3]{AIM}); there exists a homeomorphism $\eta\colon[0,\infty) \to [0,\infty)$ (called the gauge function) such that for any $a, b, x \in \mathbb{C}$, 
 \begin{equation}\label{quasisymmetry}
 \frac{|a-x|}{|b-x|} \leq t \implies \frac{|G(a)-G(x)|}{|G(b)-G(x)|} \leq \eta(t).
 \end{equation}
We apply this property to the points $a=x+h$ and $b=x+iy$ with $x, h \in \mathbb{R}$ and $y>0$. 
Then,
\begin{equation}\label{qs1}
|g(x+h)-g(x)| \leq \eta(|h|/y)|G(x+iy)-g(x)|.
\end{equation}

Moreover, we take $\xi \in \mathbb{R}$ such that $g(\xi) \in \Gamma$ is a nearest point to $G(x+iy)$.
By
\[
\frac{|x-(x+iy)|}{ |\xi-(x+iy)|}\leq 1\implies
\frac{|g(x)-G(x+iy)|}{|g(\xi)-G(x+iy)|} \leq \eta(1),
\]
we have
\begin{equation}\label{d_f}
|G(x+iy)-g(x)| \leq \eta(1)d(G(x+iy),\Gamma).
\end{equation}

Condition (2) states that there is some $y_0>0$ such that $d(G(x+iy_0),\Gamma)$ is bounded by some constant $M>0$.
Then, it follows from \eqref{qs1} and \eqref{d_f} that
\[
|g(x+h)-g(x)| \leq \eta(|h|/y_0)\eta(1) M.
\]
This proves condition (3); $g$ is uniformly continuous.

The remaining implication is proved by reversing above inequalities.
Namely, for \eqref{qs1}, we have
\begin{equation}\label{qs2}
|G(x+iy)-g(x)| \leq \eta(y/|h|)|g(x+h)-g(x)|,
\end{equation}
and for \eqref{d_f}, we have
\begin{equation}\label{trivial}
d(G(x+iy),\Gamma) \leq |G(x+iy)-g(x)|.
\end{equation}
By taking $h=y$, we see that condition (3) implies (1).
\end{proof}

\begin{remark}
The distortion theorem implies that
\begin{equation}\label{yG'}
\frac{1}{2}(\operatorname{Im} \,z)|G'(z)| \leq d(G(z),\Gamma) \leq 2(\operatorname{Im} \,z)|G'(z)|.
\end{equation}
In the case of the unit disk $\mathbb{D}$, see \cite[Corollary 1.4]{Pom}.
\end{remark}

The proof of Theorem \ref{AS->B_0} is obtained by modifying the arguments in \cite{Po78}.

\begin{proof}[Proof of Theorem \ref{AS->B_0}]
For each closed interval $I$ of length $1$ in $\mathbb{R}$, we associate a stadium domain $D_I=\bigcup_{x \in I} B(x+i/2)$ in $\mathbb{H}$, where $B(\cdot)$ denotes the open disk centered at the prescribed point of radius $1/2$. 
The boundary $\partial D_I$ is $C^1$-smooth with non-degenerate Lipschitz continuous tangent vectors with respect to the arc-length parameter. Moreover, $\partial D_I \cap \mathbb{R}=I$.
For each $I$, we fix a normalized conformal homeomorphism $T_I\colon\mathbb{D} \to D_I$ so that $T_I(0)=x_I+i/2$ and $T_I'(0)>0$, where $x_I$ is the midpoint of $I$; all maps $T_I$ differ only by translations.
Let $\tilde{I}=T_I^{-1}(I)$ be the arc in $\partial \mathbb{D}$ corresponding to $I$, which contains $-i$ as the midpoint.

We first note that $\log T'_I$ belongs to $B_0(\mathbb{D})$.
Indeed, $\partial D_I$ is an asymptotically conformal quasicircle by construction and
hence \cite[Theorem 1.1]{Po78} implies the conclusion. 
Let $G_I=G \circ T_I$ for each $I$, which is a conformal homeomorphism of $\mathbb{D}$ onto a bounded Jordan domain $\Omega_I$ in $\Omega$. Since $G$ extends to a quasiconformal self-homeomorphism of $\mathbb{C}$ and it maps the quasidisk $D_I$ onto $\Omega_I$, this is also a quasidisk and $\partial \Omega_I$ is a quasicircle. 
We have that $\partial \Omega_I \cap \Gamma=g(I)=G_I(\tilde{I})$.

Assuming that $\Gamma$ is an asymptotically conformal quasicircle, we show that $\log G' \in B_0(\mathbb{H})$.
Suppose to the contrary that $\log G' \notin B_0(\mathbb{H})$.
Then, there exists a sequence $\{\tilde z_n\}$ in $\mathbb{H}$ and a constant $c>0$ such that $\operatorname{Im} \,\tilde z_n \to 0$ as $n \to \infty$ and 
\begin{equation}\label{ass}
(\operatorname{Im} \,\tilde z_n)|(\log G')'(\tilde z_n)| \geq 2c
\end{equation} 
for all $n$.
Let $I_n$ be the closed interval of $\mathbb{R}$ with $|I_n|=1$ whose midpoint is $\operatorname{Re} \,\tilde z_n$.
We may assume that $\tilde z_n \in D_{I_n}$ and set $\tilde \zeta_n=T_{I_n}^{-1}(\tilde z_n) \in \mathbb{D}$.
The normalization of $T_{I_n}$ implies that $\arg \tilde{\zeta_n}=-\pi/2$ for all $n$ and $|\tilde \zeta_n| \to 1$ as $n \to \infty$.
The chain rule of pre-Schwarzian derivatives yields
\begin{equation}\label{chain}
(\log G'_{I_n})'=(\log G')' \circ T_{I_n} \cdot T_{I_n}'+(\log T'_{I_n})'.
\end{equation}
Then, using the comparability of the hyperbolic density with the reciprocal of the distance to the boundary established by \eqref{yG'}, we obtain
\[
(1-|\tilde \zeta_n|^2)|(\log G'_{I_n})'(\tilde \zeta_n)|\geq d(\tilde z_n,\partial D_{I_n})|(\log G')'(\tilde z_n)|-(1-|\tilde \zeta_n|^2)|(\log T'_{I_n})'(\tilde \zeta_n)|.
\]
Here, 
\[
(1-|\tilde \zeta_n|^2)|(\log T'_{I_n})'(\tilde \zeta_n)| \to 0\quad (n \to \infty)
\]
because $\log T'_{I_n} \in B_0(\mathbb{D})$,
and \eqref{ass} implies
\[
d(\tilde z_n,\partial D_{I_n})|(\log G')'(\tilde z_n)| \geq 2c
\]
for all sufficiently large $n$ because $d(\tilde z_n,\partial D_{I_n})=\operatorname{Im} \,\tilde z_n$ in these cases.
Thus,
\begin{equation}\label{contradiction}
(1-|\tilde \zeta_n|^2)|(\log G'_{I_n})'(\tilde \zeta_n)| \geq c >0
\end{equation}
for all sufficiently large $n$.

For $\zeta=\rho e^{i\theta} \in \mathbb{D}$, let 
\[
\Phi_\zeta(z)=e^{i\theta}\frac{z+\rho}{1+\rho z}, \quad \Phi'_\zeta(z)=e^{i\theta}\frac{1-\rho^2}{(1+\rho z)^2}
\]
be the conformal automorphism of $\mathbb{D}$ sending $0$ to $\zeta$ and its derivative.
For an arbitrarily fixed $r \in [0,1)$, we choose $\zeta_n \in \mathbb{D}$ such that $\tilde \zeta_n=\Phi_{\zeta_n}(r)$. 
This means that $\tilde \zeta_n$ is located at $r$ on the segment between $\zeta_n$ $(r=0)$ and $\zeta_n/|\zeta_n|=-i$ $(r=1)$ when the segment is parametrized conformally by $[0,1]$. In particular, $\arg \zeta_n=-\pi/2$ for all $n$ and $|\zeta_n| \to 1$ as $n \to \infty$.
Since 
\[
|\Phi'_{\zeta_n}(r)|=\frac{1-|\zeta_n|^2}{(1+|\zeta_n|r)^2}=\frac{1-\left(\frac{r+|\zeta_n|}{1+|\zeta_n|r}\right)^2}{1-r^2}=\frac{1-|\tilde \zeta_n|^2}{1-r^2}, 
\]
the inequality \eqref{contradiction} can be transformed into the condition that the sequence $\{\zeta_n\} \subset \mathbb{D}$ satisfies
\begin{equation}\label{another}
|\Phi'_{\zeta_n}(r)||(\log G'_{I_n})'(\Phi_{\zeta_n}(r))| \geq \frac{c}{1-r^2}.
\end{equation}
for any $r \in [0,1)$. 

We set $d(\zeta_n)=d(G_{I_n}(\zeta_n),\partial \Omega_{I_n})$, which is bounded by $d(G_{I_n}(\zeta_n),\Gamma)$.
By the assumption that $g$ is uniformly continuous, Lemma \ref{uc1} implies that $d(\zeta_n)$ tends to $0$ as $n \to \infty$. 
Passing to a subsequence if necessary, we may assume that
$nd(\zeta_n) \to 0$. 
Moreover, after passing a further subsequence, we have that $\partial \Omega_{I_n} \setminus \Gamma$ has no intersection with the closed disk
\[
B_n=\{w \in \mathbb{C} : |w-G_{I_n}(\zeta_n)| \leq nd(\zeta_n)\}.
\]
Indeed, since $G$ extends to a quasiconformal self-homeomorphism of $\mathbb{C}$, it is quasisymmetric
satisfying \eqref{quasisymmetry} with the gauge function $\eta$. Let $z_n=T_{I_n}(\zeta_n)$ and take an arbitrary point $\xi=G^{-1}(w)$ in $\partial D_{I_n} \setminus I_n$. If $|x_n-z_n| \leq t|\xi-z_n|$ then
\[
d(\zeta_n) \leq |G(x_n)-G(z_n)| \leq \eta(t)|G(\xi)-G(z_n)|=\eta(t)|w-G_{I_n}(\zeta_n)|.
\]
We choose $t_n>0$ such that $\eta(t_n)<1/n$. Since $|z_n-x_n| \to 0$ as $n \to \infty$, we may assume that
\[
|z_n-x_n| \leq t_n/2 \leq t_n|\xi-z_n|.
\]
Then, $nd(\zeta_n)<|w-G(x_n)|$ for every $w \in G(\partial D_{I_n} \setminus I_n)=\partial \Omega_{I_n} \setminus \Gamma$.

Let $\Gamma_n$ be the smallest closed subarc of $\Gamma$ that contains $\Gamma \cap B_n$.

Let $w_n$ and $w'_n$ be the endpoints of $\Gamma_n$, which satisfy $|w_n-G_{I_n}(\zeta_n)| = nd(\zeta_n)$
and $|w'_n-G_{I_n}(\zeta_n)| = nd(\zeta_n)$. 
In particular, $|w_n-w'_n| \leq 2nd(\zeta_n)$.
The assumption that $\Gamma$ is an asymptotically conformal quasicircle implies that
\begin{equation}\label{n4}
\max_{w \in \Gamma_n} \frac{|w_n-w|+|w-w'_n|}{|w_n-w'_n|} <1+\frac{1}{n^4}
\end{equation}
for a further subsequence.

Let $\alpha_n=\arg(w_n-w'_n)$ and define an affine transformation of $\mathbb{C}$ by
\[
\Psi_n(w)=\frac{e^{-i\alpha_n}}{d(\zeta_n)}(w-G_{I_n}(\zeta_n)).
\]
For each $n$, this means $|\Psi_n(w_n)|=|\Psi_n(w'_n)|=n$ and $\Psi_n(w_n)-\Psi_n(w'_n) >0$. 
Let $\Psi_n(w_n)=u_n+iv_n$ and $\Psi_n(w'_n)=-u_n+iv_n$ for $u_n>0$. 
We consider the affine image $\Psi_n(\Gamma_n)$, which satisfies the same condition as \eqref{n4}:
\[
\max_{\tilde w \in \Psi_n(\Gamma_n)} \,\frac{|\Psi_n(w_n)-\tilde w|+|\tilde w-\Psi_n(w'_n)|}{|\Psi_n(w_n)-\Psi_n(w'_n)|} 
<1+\frac{1}{n^4}.
\]
Then, by arguments similar to those in \cite[(2.6), (2.7)]{Po78} based on these properties, the following claims hold. 

\medskip
\noindent
{\it Claim 1.} $\Psi_n(\Gamma_n) \subset \{\tilde w \in \mathbb{C} :|\operatorname{Im} \,\tilde w-v_n|<2u_n n^{-2}\}$.

\medskip
\noindent
{\it Claim 2.} $u_n/n \to 1$ and $v_n \to 1$ as $n \to \infty$.

\medskip

We note that Claims 1 and 2 are obtained straightforwardly from the original proof. 
Moreover, the condition $(\partial \Omega_{I_n} \setminus \Gamma) \cap B_n=\emptyset$ implies that 
\begin{equation}\label{claim}
\Psi_n(\partial \Omega_{I_n} \setminus \Gamma_n) \subset \{\tilde w \in \mathbb{C} :|\tilde w|>n\}.
\end{equation}

We consider the composition
\begin{equation}\label{composition}
\tilde G_n(z)=\Psi_n \circ G_{I_n} \circ \Phi_{\zeta_n}(z)=\frac{G_{I_n}(\Phi_{\zeta_n}(z))-G_{I_n}(\zeta_n)}{e^{-i\alpha_n}d(\zeta_n)},
\end{equation}
which satisfies $\tilde G_n(0)=0$. Moreover, since
\[
\tilde G_n'(0)=-i\frac{1-|\zeta_n|^2}{e^{-i\alpha_n}d(\zeta_n)}G_{I_n}'(\zeta_n),
\]
we have $1/4 \leq |\tilde G_n'(0)| \leq 1$ by \eqref{yG'} in the disk case.

By Claims 1 and 2 together with \eqref{claim}, we see that the sequence of domains $\tilde G_n(\mathbb{D})$ converges to the kernel with respect to $0$, which is 
\[
{\mathbb{H}}^*(1)=\{w \in \mathbb{C} : \operatorname{Im} \,w <1\}. 
\]
We may assume that the limit $\beta=\lim_{n \to \infty}\,\arg \tilde G_n'(0)$ exists by passing to a subsequence.
Then, by the Carath\'eodory kernel convergence theorem, $\tilde G_n$ converges locally uniformly to a conformal mapping $\tilde G$ of $\mathbb{D}$ onto ${\mathbb{H}}^*(1)$. 
In addition, $\tilde G$ satisfies $\tilde G(0)=0$ and $\arg \tilde G'(0)=\beta$. 

For $b=ie^{i\beta}$, $\tilde G$ is the M\"obius transformation precisely determined as
\[
\tilde G(z)=-\frac{2ibz}{1-bz}.
\]
This function satisfies the following property:

\medskip
\noindent
{\it Claim 3.} There exists a constant $K>0$ such that $|\tilde G_n(r)| \leq K$ for any $n \in \mathbb N$,
and hence
\[
|\tilde G(r)|=\left|\frac{2br}{1-br}\right| \leq K
\]
for any $r \in [0,1)$. 
\medskip

Assuming this claim for the moment, we finish the proof by deriving a contradiction.
The chain rule of pre-Schwarzian derivatives for $\tilde G_n=\Psi_n \circ G_{I_n} \circ \Phi_{\zeta_n}$ gives
\[
(\log \tilde G'_n)'(z)=(\log G'_{I_n})'(\Phi_{\zeta_n}(z))\Phi'_{\zeta_n}(z)+(\log \Phi'_{\zeta_n})'(z),
\]
and from this it follows that
\begin{equation}\label{real}
   \begin{aligned}
   |\Phi'_{\zeta_n}(r)||(\log G'_{I_n})'(\Phi_{\zeta_n}(r))|&=
   |(\log \tilde G'_n)'(r)-(\log \Phi'_{\zeta_n})'(r)|\\
   &=\left|(\log \tilde G'_n)'(r)+\frac{2|\zeta_n|}{1+|\zeta_n| r}\right| 
   \end{aligned}
\end{equation}
for any $r \in [0,1)$. Here, \eqref{another} claims that the left-hand side term of \eqref{real} is bounded below by $c/(1-r^2)$.
On the other hand, the limit of the right side term is
\[
\lim_{n \to \infty}\,\left|(\log\tilde G'_n)'(r)+\frac{2|\zeta_n|}{1+|\zeta_n| r}\right|
=\left|(\log \tilde G')'(r)+\frac{2}{1+r}\right|=\left|\frac{2b}{1-br}+\frac{2}{1+r}\right|.
\]
This quantity is bounded uniformly for $r \in [0,1)$ since $b=ie^{i\beta}$ is away from $1$ on the unit circle by Claim 3.
However, this contradicts the existence of the lower bound $c/(1-r^2)$.

Finally, we verify Claim 3. 
By the Kellogg--Warschawski theorem (see \cite[Theorem 3.6]{Pom}), $T_{I_n}$ extends to a $C^1$-diffeomorphism of $\overline{\mathbb{D}}$ onto $\overline{D_{I_n}}$, and in particular $|T_{I_n}'|$ is bounded above and bounded away from $0$. 
Namely, $T_{I_n}$ is bi-Lipschitz 
with a uniform constant $\kappa \geq 1$ satisfying $\kappa^{-1} \leq |T'_{I_n}(z)| \leq \kappa$ for $z \in \overline{\mathbb{D}}$. 
The conformal mapping $G\colon\mathbb{H} \to \Omega$ extends to a quasisymmetric homeomorphism of $\mathbb{C}$ with the gauge function $\eta(t)$. 
Hence, $G_{I_n}\colon\overline{\mathbb{D}} \to \overline{\Omega_{I_n}}$ is
quasisymmetric with the gauge function $\eta(\kappa^2t)$. 

We apply the quasisymmetry of $G_{I_n}$ to $\zeta_n$, $\tilde \zeta_n$, and $-i$ in $\overline{\mathbb{D}}$.
Since $|\tilde \zeta_n-\zeta_n| \leq |-i-\zeta_n|$, we have
\begin{equation}\label{kappa1}
|G_{I_n}(\tilde \zeta_n)-G_{I_n}(\zeta_n)| \leq \eta(\kappa^2)|G_{I_n}(-i)-G_{I_n}(\zeta_n)|.
\end{equation}
Take $e^{i\theta_n} \in \partial \mathbb{D}$ such that $G_{I_n}(e^{i\theta_n})$ is a nearest point on $\partial \Omega_{I_n}$ to $G_{I_n}(\zeta_n)$. 
This satisfies $|G_{I_n}(e^{i\theta_n})-G_{I_n}(\zeta_n)| = d(\zeta_n)$. 
Then, by the quasisymmetry of $G_{I_n}$ applied to $\zeta_n$, $e^{i\theta_n}$, and $-i$, we have
\begin{equation}\label{kappa2}
|G_{I_n}(-i)-G_{I_n}(\zeta_n)|\leq \eta(\kappa^2) |G_{I_n}(e^{i\theta_n})-G_{I_n}(\zeta_n)|.
\end{equation}

By \eqref{kappa1} and \eqref{kappa2}, we obtain that
\[
|G_{I_n}(\tilde \zeta_n)-G_{I_n}(\zeta_n)| \leq \eta(\kappa^2)^2 d(\zeta_n).
\]
Since $\tilde \zeta_n=\Phi_{\zeta_n}(r)$, the definition \eqref{composition} of $\tilde G_n$ gives
\[
|\tilde G_n(r)|=\left|\frac{G_{I_n}(\Phi_{\zeta_n}(r))-G_{I_n}(\zeta_n)}{d(\zeta_n)}\right| \leq \eta(\kappa^2)^2.
\]
By taking $K=\eta(\kappa^2)^2$, Claim 3 is proved.
\end{proof}

The preceding argument also yields a local version of this theorem, which does not require the uniform continuity of $g$.

\begin{corollary}\label{localization}
If $\Gamma$ is an asymptotically conformal quasicircle then $\log G'$ belongs to $B_0^{\mathrm{loc}}(\mathbb{H})$.
\end{corollary}

The study of subarcs of unbounded asymptotically conformal quasicircles can be reduced to the bounded case by localization. 
Formally, a subarc of an unbounded quasicircle can be realized as a subarc of a bounded quasicircle as shown in the following result:

\begin{proposition}\label{localize_Gamma:label}
Suppose that $G \colon \mathbb{H} \to \Omega$ is a conformal mapping such that $\log G' \in B_0^{\mathrm{loc}}(\mathbb{H})$. 
Then for every interval $I \subset \mathbb{R}$, there exists a conformal mapping $T_I\colon \mathbb{D}\to\mathbb{H}$ such that
\[
\log (G \circ T_I)' \in B_0(\mathbb{D}).
\]
\end{proposition}

\begin{proof}
Following the approach in the proof of Theorem~\ref{AS->B_0}, for each closed interval $I\subset\mathbb{R}$, we associate the stadium domain $D_I=\bigcup_{x \in I} B(x+i/2)$ in $\mathbb{H}$ as before. 
For each $I$, we fix a normalized conformal homeomorphism $T_I\colon\mathbb{D} \to D_I$ so that $T_I(0)=(x_I,i/2)$ and $T_I'(0)>0$, where $x_I$ is the midpoint of $I$. 
Let $\tilde{I}=T_I^{-1}(I)$ be the arc in $\partial \mathbb{D}$ corresponding to $I$, which contains $-i$ as the midpoint. 
Moreover, $\log T_I'$ belongs to $B_0(\mathbb{D})$ and $T_I$ is bi-Lipschitz.

The chain rule~\eqref{chain} for the pre-Schwarzian derivative together with the comparability of the hyperbolic density gives
\begin{equation}\label{B0condition}
\begin{aligned}
&\quad (1-|z|^2)|(\log (G\circ T_I)')'(z)| \\
&\leq 2d(T_I(z),\partial D_I) |(\log G')' (T_I(z))|+(1-|z|^2)|(\log T_I')'(z)|.
\end{aligned}
\end{equation}
Since $\log T_I' \in B_0(\mathbb{D})$, it suffices to control the first term on the right-hand side.

Fix any $\varepsilon>0$. By the assumption $\log G' \in B_0^{\mathrm{loc}}(\mathbb{H})$, there exists some $\delta_1>0$ such that
if $\operatorname{Im}w <\delta_1$ then
\[
 \sup_{\operatorname{Re}w\in J}\, (\operatorname{Im}w) |(\log G')'(w)|<\varepsilon,
\]
where $J$ is the interval with the same center as $I$ and length $|J| = |I| + 1$. 
Since $d(T_I(z),\partial D_I) \leq \operatorname{Im} T_I(z)$, it follows that whenever $\operatorname{Im} T_I(z) < \delta_1$,
\begin{equation}\label{epsilon_logGF}
d(T_I(z),\partial D_I)\, |(\log G')'(T_I(z))| < \varepsilon.
\end{equation}

On the other hand, by Proposition~\ref{bloch_norm_H:label}, for all $\zeta \in \mathbb{H}$ with $\operatorname{Im} \zeta \geq \delta_1$, we have $|(\log G')' (\zeta)| \leq 3/\delta_1$.

Hence, if $d(T_I(z),\partial D_I) < \delta \coloneqq \delta_1 \varepsilon / 3$, then \eqref{epsilon_logGF} holds. 
Combining this with \eqref{B0condition} and the fact that $\log T_I'\in B_0(\mathbb D)$, we obtain 
\[
(1-|z|^2)\big|(\log (G\circ T_I)')'(z)\big| \to 0\quad \text{as } d(T_I(z),\partial D_I) \to 0.
\]

Finally, by the Koebe distortion theorem and the bi-Lipschitz property of $T_I$, the condition $d(T_I(z),\partial D_I) \to 0$ is equivalent to $|z| \to 1$. 
Therefore,
\[
(1-|z|^2)\big|(\log (G\circ T_I)')'(z)\big| \to 0 \quad \text{as } |z| \to 1,
\]
which shows that $\log (G \circ T_I)' \in B_0(\mathbb{D})$.
\end{proof}

\section{Little Bloch implies asymptotic conformality}\label{sec:little_bloch_implies_asymptotic_conformality}

The converse of Theorem \ref{AS->B_0} is also obtained under the additional assumption that $g^{-1}$ is uniformly continuous.
The proof follows a modification of \cite{Po78} after establishing a lemma complementary to Lemma~\ref{uc1}.

\begin{theorem}\label{B_0->AS}
Under the assumption that $g^{-1}$ is uniformly continuous, if $\log G' \in B_0(\mathbb{H})$ then $\Gamma$ is an asymptotically conformal quasicircle. 
\end{theorem}

\begin{lemma}\label{uc2}
The following conditions are equivalent:
   \begin{enumerate}
   \item
   For any $\varepsilon>0$, there exists some $\delta>0$ such that
   if $y \geq \varepsilon$ then $d(G(x+iy),\Gamma) \geq \delta$;
   \item
   $d(G(x+iy),\Gamma)$ is uniformly bounded away from $0$ for some $y>0$;
   \item
   $g^{-1}$ is uniformly continuous.
   \end{enumerate}
\end{lemma}

\begin{proof}
It suffices to prove that $(2) \Rightarrow (3)$ and $(3) \Rightarrow (1)$ since $(1) \Rightarrow (2)$ is obvious. 

As in the proof of Lemma \ref{uc1}, we apply the inequalities \eqref{qs2} and \eqref{trivial} to deduce that 
\[
|g(x+h)-g(x)| \geq \frac{1}{\eta(y/|h|)}|G(x+iy)-g(x)| \geq \frac{1}{\eta(y/|h|)}d(G(x+iy),\Gamma).
\]
Condition (2) says that there is some $y_0>0$ such that $d(G(x+iy_0),\Gamma)$ is bounded below by some constant $m>0$.
Denote $g(x+h)=\xi_1$ and $g(x)=\xi_2$, we have that if $|h| \geq \varepsilon$, then
\[
|\xi_1-\xi_2| \geq \frac{m}{\eta(y_0/|h|)} \geq \frac{m}{\eta(y_0/\varepsilon)}.
\]
The contraposition claims that for any $\varepsilon>0$ if $|\xi_1-\xi_2| < \frac{m}{\eta(y_0/\varepsilon)}$ 
then $|g^{-1}(\xi_1)-g^{-1}(\xi_2)|<\varepsilon$.
This proves condition (3).

The implication $(3)\Rightarrow(1)$ follows by reversing the above argument and applying \eqref{qs1} and \eqref{d_f}.
\end{proof}

\begin{remark}
For condition (1) above, by replacing the constant $\delta$ with $\delta(\varepsilon)=\min\{\delta,\varepsilon\}$, we may assume that $\delta(\varepsilon) \to 0$ as $\varepsilon \to 0$.
\end{remark}

\begin{proof}[Proof of Theorem \ref{B_0->AS}]
Assume that $\log G' \in B_0(\mathbb{H})$. As in the proof of Theorem \ref{AS->B_0}, for each closed interval $I \subset \mathbb{R}$ of length $1$, we consider $T_I\colon \mathbb{D} \to D_I \subset \mathbb{H}$ and $G_I=G \circ T_I$.
The same argument as in Proposition~\ref{localize_Gamma:label} yields that $\log G'_I \in B_0(\mathbb{D})$.

Moreover, this vanishing condition is uniform in $I$. 
Namely, for any $\tilde \varepsilon >0$ and any $I \subset \mathbb{R}$ with $|I|=1$, there exists some $\tilde \delta>0$ such that if $1-|z|<\tilde \delta$ then $(1-|z|^2)|(\log G'_I)'(z)|<\tilde \varepsilon$.

Under this condition, as shown in the proof of the implication i) $\Rightarrow$ ii) of \cite[Theorem 1]{Po78}, we obtain the following:
For any $\varepsilon >0$, any $I \subset \mathbb{R}$ with $|I|=1$, and any $z \in \overline{\mathbb{D}}$, there exists some $\delta>0$ such that if $1-|\zeta|<\delta$ and $|z-\zeta| \leq 2(1-|\zeta|)$ then the Visser--Ostrowski quotient (see \cite[Section 11.3]{Pom} for more information) satisfies
\begin{equation}\label{VO}
\left|\frac{G_I(z)-G_I(\zeta)}{(z-\zeta)G'_I(\zeta)}-1 \right|<\varepsilon.
\end{equation}
Here, by the linear approximation of $G_I(z)$ at $\zeta$, we have
\begin{equation}\label{linear}
G_I(z)=G_I(\zeta)+(z-\zeta)G'_I(\zeta)(1+R_I^\zeta(z)),
\end{equation}
where $R_I^\zeta(z)$ is the remainder term tending to $0$ as $z \to \zeta$. 
By \eqref{VO}, $R_I^\zeta(z) \to 0$ uniformly in $|z-\zeta| \leq 2(1-|\zeta|)$ as $|\zeta| \to 1$, independently of $I$. The Taylor expansion of $R_I^\zeta(z)$ at $\zeta$ is given by
\begin{equation}\label{Taylor}
R_I^\zeta(z)=\frac{1}{2}\frac{G''_I(\zeta)}{G'_I(\zeta)}(z-\zeta)+\frac{1}{6}\frac{G'''_I(\zeta)}{G'_I(\zeta)}(z-\zeta)^2+\cdots.
\end{equation}

Take $w_1$ and $w_2$ on $\Gamma$. We may assume that there exists some closed interval $I \subset \mathbb{R}$
with $|I|=1$ such that $w_1, w_2 \in g(I)$.
Take $z_1, z_2 \in \partial \mathbb{D}$ such that $G_I(z_1)=w_1$ and $G_I(z_2)=w_2$.
Choose $\zeta \in \mathbb{D}$ such that $|z_i-\zeta|=2(1-|\zeta|)$ for $i=1,2$. 
We choose any $z_0$ in the arc $\wideparen{z_1z_2}$ and let $w_0=G_I(z_0)$.
Then, \eqref{linear} for $z=z_i$ $(i=0,1,2)$ is written as
\[
w_i-G_I(\zeta)=(z_i-\zeta)G'_I(\zeta)(1+R_I^\zeta(z_i)).
\]
Taking the differences of these equations yields 
\begin{equation}\label{ij}
w_i-w_j=(z_i-z_j)G'_I(\zeta)(1+\tilde R_I^\zeta(z_i,z_j)) \quad (i \neq j),
\end{equation}
where $\tilde R_I^\zeta(z_i,z_j) \to 0$ as $|z_1-z_2| \to 0$ whenever $\zeta$ is taken as $|z_i-\zeta|=2(1-|\zeta|)$ and no matter how $z_0 \in \wideparen{z_1z_2}$ is chosen, and independently of $I$.
In fact, \eqref{Taylor} gives that
\[
\tilde R_I^\zeta(z_i,z_j)=\frac{1}{2}\frac{G''_I(\zeta)}{G'_I(\zeta)}(z_i+z_j-2\zeta)+o(1) \quad (|z_i-z_j| \to 0).
\]
By \eqref{ij}, we have
\[
\frac{|w_1-w_0|+|w_0-w_2|}{|w_1-w_2|}=\frac{|z_1-z_0|+|z_0-z_2|+o(|z_1-z_2|)}{|z_1-z_2|+o(|z_1-z_2|)} \to 1
\]
uniformly as $|z_1-z_2| \to 0$. 
The assumption that $g^{-1}$ is uniformly continuous ensures that $|z_1-z_2| \to 0$ uniformly as $|w_1-w_2| \to 0$. 
This proves that $\Gamma$ is an asymptotically conformal quasicircle.
\end{proof}

A localized version of Theorem~\ref{B_0->AS} also holds. 
We say that a subarc of $\Gamma$ is asymptotically conformal if points $a$, $b$, and $w$ on the subarc satisfy
the condition for the asymptotic conformality \eqref{AC}.

\begin{corollary}\label{localization2}
If $\log G' \in B_0(\mathbb{H})$ then $\Gamma$ is locally asymptotically conformal, that is, any bounded closed subarc of $\Gamma$ is asymptotically conformal.
\end{corollary}

\section{Asymptotically symmetric embeddings}\label{sec:asymptotically_symmetric_embeddings}

In the previous two sections, we have established the equivalence between the little Bloch condition and the asymptotic conformality under the assumption that the boundary map is uniformly continuous.
In this section, we establish their equivalence without these uniform continuity assumptions.
In this section, we formulate their equivalence without this assumption.

We continue with the conformal mapping $G\colon \mathbb{H} \to \Omega$ with $G(\infty)=\infty$ whose boundary extension $g\colon \mathbb{R} \to \Gamma \subset \mathbb{C}$ is quasisymmetric. 
The crucial step is Theorem \ref{import}, which shows that $\log G' \in B_0(\mathbb{H})$ implies that $g$ is an asymptotically symmetric embedding. 

First, we extend the results of Brania and Yang \cite[Theorems 3.1 \& 3.2]{BY} to the present non-compact setting. 

\begin{proposition}\label{BYmain}
Suppose that $\Gamma$ is an asymptotically conformal quasicircle.
Under the assumption that $g\colon \mathbb R\to\Gamma$ is uniformly continuous, the map $g$ is an asymptotically symmetric embedding.
\end{proposition}

\begin{proof}
By Theorem \ref{AS->B_0}, we have $\log G' \in B_0(\mathbb{H})$. 
Then, by the half-plane version of a theorem of Becker and Pommerenke \cite{BP} (see \cite[Theorem 5.1]{MW}), $G$ extends to a quasiconformal self-homeomorphism of $\mathbb{C}$ that is asymptotically conformal on the lower half-plane $\mathbb{H}^*$.
Consequently, the proof of \cite[Theorem 3.1]{BY}, based on modulus estimates for curve families, works in the present setting as well, which implies that $g=G|_{\mathbb{R}}$ is an asymptotically symmetric embedding.
\end{proof}

Using the relative notion introduced in Section~\ref{sec:introduction},
we consider the asymptotic conformality of $\Gamma$.
The following claim is obvious from the definition.

\begin{proposition}\label{general}
Let $\gamma\colon \mathbb{R} \to \mathbb{C}$ be a quasisymmetric embedding and suppose that its image $\Gamma=\gamma(\mathbb{R})$ is an asymptotically conformal quasicircle.
If $\gamma$ is uniformly continuous then $\Gamma$ is asymptotically conformal relative to the parametrization $\gamma\colon \mathbb{R} \to \Gamma$.
\end{proposition}


The proof of \cite[Theorem 3.1]{BY},
combined with the argument in Proposition~\ref{BYmain}, also yields the following theorem.
Indeed, the same proof applies to the relative formulation of asymptotic conformality, which is weaker than the assumption in Proposition~\ref{BYmain} by Proposition~\ref{general}.

\begin{theorem}
If $\Gamma$ is an asymptotically conformal quasicircle relative to the parametrization $g\colon \mathbb{R} \to \Gamma$, then $g$ is an asymptotically symmetric embedding.
\end{theorem}

The converse statement holds more generally for asymptotically symmetric embeddings, without assuming that they arise as boundary extensions of conformal mappings. 
This extends \cite[Theorem 3.2]{BY} to the present setting.
The same proof works both when we assume that $\gamma^{-1}$ is uniformly continuous and when we assume that $\Gamma$ is asymptotically conformal relative to $\gamma\colon \mathbb{R} \to \Gamma$.

\begin{theorem}\label{BY2}
Let $\gamma\colon \mathbb{R} \to \mathbb{C}$ be an asymptotically symmetric embedding with image $\Gamma=\gamma(\mathbb{R})$. 
Then $\Gamma$ is asymptotically conformal relative to $\gamma\colon \mathbb{R} \to \Gamma$.
Moreover, under the assumption that $\gamma^{-1}$ is uniformly continuous, $\Gamma$ is an asymptotically conformal quasicircle. 
\end{theorem}

We now revisit the proof of Theorem~\ref{AS->B_0}.
Theorem~\ref{AS->B_0} remains valid if the uniform continuity assumption on $g$ is replaced by the relative asymptotic conformality of $\Gamma$.

\begin{theorem}\label{newformilation}
If $\Gamma$ is an asymptotically conformal quasicircle relative to $g\colon \mathbb{R} \to \Gamma$, then $\log G' \in B_0(\mathbb{H})$.
\end{theorem}

\begin{proof}
In the proof of Theorem \ref{AS->B_0}, it is no longer guaranteed that $nd(\zeta_n)\to 0$ as $n\to\infty$ without the uniform continuity of $g$, since the localization scale may drift along the boundary. 
Nevertheless, the estimate \eqref{n4} still follows from the assumption that $\Gamma$ is asymptotically conformal relative to $g\colon \mathbb{R} \to \Gamma$. 
Note that the configurations of the three points $G_{I_n}(\zeta_n)$, $w_n$, $w'_n$ and the corresponding three points $z_n$, $G^{-1}(w_n)$, $G^{-1}(w'_n)$ are comparable by the quasisymmetric self-homeomorphism $G$ of $\mathbb{C}$.
The rest of the proof proceeds as in Theorem~\ref{AS->B_0} and yields the desired conclusion.
\end{proof}

Finally, we prove the main result of this section.

\begin{theorem}\label{import}
Let $G\colon \mathbb{H}\to \Omega$ be a conformal mapping onto an unbounded quasidisk $\Omega$ with $G(\infty)=\infty$ such that
\[
\log G' \in B_0(\mathbb{H}).
\]
Let $g\colon \mathbb{R}\to \Gamma=\partial\Omega$ be the boundary extension of $G$.
Then $g$ is asymptotically symmetric.
\end{theorem}

\begin{proof}
Fix $0<\varepsilon <1/10$. Since $\log G'\in B_0(\mathbb{H})$, there exists $\delta_0>0$ such that
\[
(\operatorname{Im}z )\left|\frac{G''}{G'}(z)\right|\leq \varepsilon 
\]
whenever $0<\operatorname{Im}z<\delta_0$. 
The facts that $G'\neq0$ in $\mathbb{H}$ and $\mathbb{H}$ is simply connected imply that we may choose a holomorphic branch of $\log G'(z)$ on $\mathbb{H}$.

\medskip
Let $a,b,x\in\mathbb{R}$, and let $t>0$ satisfy
\[
\frac{|x-a|}{|x-b|}\leq t.
\] 
By symmetry, it suffices to consider the case $a<x<b$. The other cases reduce to this by relabeling and replacing $t$
by $t-1$ or $t/(1-t)$, as appropriate. Assume that
\[
a<x<b\quad\text{and}\quad|x-b|=h \leq \delta_0.
\]

We first estimate $G'$ throughout the rectangle $R=[a,b]\times(0,h]$. Let $z=\xi+iu\in R$. 
Joining $w=x+ih$ to $z$ along a path consisting of a horizontal segment followed by a vertical segment, we obtain
\[
\log G'(\xi+iu)-\log G'(x+ih)=\int_x^\xi \frac{G''}{G'}(s+ih)\,ds+i\int_h^u \frac{G''}{G'}(\xi+iv)\,dv .
\]
Hence
\begin{align}\label{logG'_control}|\log G'(\xi+iu)-\log G'(x+ih)|&\leq\int_x^\xi \frac{\varepsilon}{h}\,ds+\int_u^h \frac{\varepsilon}{v}\,dv = \varepsilon\frac{|\xi-x|}{h}+\varepsilon\log\frac{h}{u}.
\end{align}
Therefore,
\begin{equation}\label{eq:Gprime_control}
\left|\frac{G'(\xi+iu)}{G'(x+ih)}\right|\leq e^{\frac{|\xi-x|}{h}\varepsilon}\left(\frac{h}{u}\right)^\varepsilon.
\end{equation}

We now estimate $g(y)-g(x)$ for arbitrary $x,y\in [a,b]$. 
For $0<h_0<h$, applying Cauchy's theorem to the rectangle $[x,y]\times[h_0,h]$ gives
\begin{equation}\label{eq:rectangle_eta}
 \begin{aligned}
 G(y+i h_0)-G(x+i h_0)=&\int_x^y G'(s+ih)\,ds+i\int_{h_0}^h G'(x+iu)\,du \\&-i\int_{h_0}^h G'(y+iu)\,du .
 \end{aligned}
\end{equation}

By the inequality \eqref{eq:Gprime_control}, 
\[
|G'(x+iu)|\leq|G'(x+ih)|\left(\frac{h}{u}\right)^\varepsilon ,\qquad u\in (0,h]
\]
and similarly 
\[
|G'(y+iu)|\leq|G'(x+ih)|e^{\frac{|y-x|}{h}\varepsilon}\left(\frac{h}{u}\right)^\varepsilon ,\qquad u\in (0,h].
\]
Since $\varepsilon<1/10$, it follows that $u^{-\varepsilon}\in L^1([0,h])$. 
Thus the vertical integrals converge absolutely as $h_0\to 0_+$.

Since $G$ extends continuously to $\mathbb{R}$ with the boundary value $g$, letting $h_0\to0_+$ in \eqref{eq:rectangle_eta} yields
\begin{equation}\label{eq:boundary_increment}
 \begin{aligned}
g(y)-g(x)&=\int_{\Gamma_{x,y}}G'(z)dz,
\end{aligned}
\end{equation}
where $\Gamma_{x,y}$ is the polygonal path
\[
x\to x+ih\to y+ih\to y .
\]

Define
\[
\theta(z)\coloneqq\frac{G'(z)}{G'(x + ih)}-1=e^{\log G'(z)-\log G'(x + ih)}-1, \qquad z=\xi+iu\in R .
\]
Together with the inequality $|e^z-1|\leq e^{|z|}-1$ and~\eqref{logG'_control}, this implies
\begin{equation}\label{theta_z_est}
 |\theta(z)|\leq e^{\tfrac{|\xi-x|}{h}\varepsilon}\left(\frac{h}{u}\right)^\varepsilon-1 .
\end{equation}

Substituting
\[
G'(z)=G'(x + ih)(1+\theta(z))
\]
into \eqref{eq:boundary_increment}, we obtain
\begin{equation}\label{gy-gx}
 \begin{aligned}
 g(y)-g(x)&=G'(x + ih)\left(\int_{\Gamma_{x,y}} dz+\int_{\Gamma_{x,y}}\theta(z)\,dz\right)\\
 &=G'(x + ih)\bigl((y-x)+E(x,y)\bigr),
 \end{aligned}
\end{equation}
where
\[
E(x,y)=\int_{\Gamma_{x,y}}\theta(z)\,dz .
\]

The error term consists of contributions from horizontal and vertical segments. Using the estimate~\eqref{theta_z_est} for the error term $E(x,y)$, and estimating the horizontal segment and the two vertical segments separately, it follows that
\begin{equation}\label{eq:error_estimate}
|E(x,y)|\leq |y-x|\left( e^{\tfrac{|y-x|}{h}\varepsilon}-1\right)+h\left(
\frac{e^{\tfrac{|y-x|}{h}\varepsilon}+1}{1-\varepsilon}-2\right).
\end{equation}

Applying \eqref{gy-gx} to the pairs $(x,a)$ and $(x,b)$, we obtain
\begin{equation}\label{gx-ga_gx-gb}
 \begin{aligned}
|g(x)-g(a)|&\leq|G'(x + ih)|\left(|x-a|+|E(x,a)|\right), \\
|g(x)-g(b)|&\geq|G'(x + ih)|\left(|x-b|-|E(x,b)|\right).
\end{aligned}
\end{equation}

Since $|x-a|\leq t|x-b|$ and $|x-b|= h$, it follows from \eqref{eq:error_estimate} that
\begin{equation}\label{Exa_Exb}
 \begin{aligned}
|E(x,a)|&\leq th(e^{t\varepsilon}-1)+h\left(\tfrac{e^{t\varepsilon}+1}{1-\varepsilon}-2\right)\leq h\left(t(e^{t\varepsilon}-1)+(e^{t\varepsilon}+1)(1+2\varepsilon)-2\right), \\
|E(x,b)|&\leq h\left((e^{\varepsilon}-1)+\left(\tfrac{e^{\varepsilon}+1}{1-\varepsilon}-2\right)\right)\leq h\left(2\varepsilon+\frac{1+2\varepsilon+1}{1-\varepsilon}-2\right)\leq 5h\varepsilon.
\end{aligned}
\end{equation}

Combining the estimates~\eqref{gx-ga_gx-gb} and~\eqref{Exa_Exb}, we obtain
\[
\frac{|g(x)-g(a)|}{|g(x)-g(b)|}\leq\frac{te^{t\varepsilon}+(e^{t\varepsilon}+1)(1+2\varepsilon)-2}{1-5\varepsilon}.
\]
Fix $t$ and $\eta>0$. Choose $\varepsilon=\min\{\tfrac{\eta t}{20(1+t^2)}, \tfrac{1}{30(1+t)}\}$ so that
\begin{equation*}
 \begin{aligned}
 &\frac{te^{t\varepsilon}+(e^{t\varepsilon}+1)(1+2\varepsilon)-2}{1-5\varepsilon}=\frac{t+(t+2\varepsilon+1)(e^{t\varepsilon}-1)+4\varepsilon}{1-5\varepsilon}\\
\leq& (1+6\varepsilon)t+2(2t\varepsilon(t+2)+4\varepsilon)= t+\varepsilon(4t^2+14t+8)\\
\leq& t+\tfrac{\eta t(4t^2+14t+8)}{20(1+t^2)}\leq(1+\eta)t .
 \end{aligned}
\end{equation*}
For this $\varepsilon$, we choose $\delta_0>0$ such that
\[
(\operatorname{Im} z )\left|\frac{G''}{G'}(z)\right|\leq \varepsilon
\]
whenever $\operatorname{Im} z<\delta_0$.

Therefore, for every $\eta>0$ and $t>0$, there exists $\delta_0>0$ such that whenever $|x-b|\leq\delta_0$ and $|x-a|\leq t|x-b|$ with $a<x<b$, we have
\begin{align}\label{verify_g}
 \frac{|x-a|}{|x-b|}\leq t\implies\frac{|g(x)-g(a)|}{|g(x)-g(b)|}\leq(1+\eta)t.
\end{align}
Consequently, if $a<x<b$ and the points $a,x,b$ lie in a ball of radius $\delta=\delta_0/2$ then \eqref{verify_g} holds. The other cases are similar. Hence $g$ is asymptotically symmetric.
\end{proof}

\begin{remark}
Theorem \ref{B_0->AS} is an immediate consequence of Theorem \ref{import} and Theorem \ref{BY2}.
In the previous proof of Theorem \ref{B_0->AS}, we follow the original idea by Pommerenke \cite{Po78}, which proves the asymptotic conformality of $\Gamma$ by examining the Visser--Ostrowski quotient; this method will be revisited in the proof of Theorem \ref{ASmooth->VMOA}. 
In contrast, the proof of Theorem \ref{import} provided above offers a more direct approach to this result.
\end{remark}

With the preceding preparations, we can now give the proof of Theorem~\ref{thm:intro-little-bloch}, which is the half-plane version of Pommerenke's theorem.

\begin{proof}[Proof of Theorem~\ref{thm:intro-little-bloch}]
Implication $(3) \Rightarrow (2)$ is proved by Theorem \ref{BY2}, and $(2) \Rightarrow (1)$ is proved by Theorem \ref{newformilation}.
Finally, $(1) \Rightarrow (3)$ is proved by Theorem \ref{import}.
\end{proof}

Combining Theorem~\ref{thm:intro-little-bloch} and \cite[Theorem 6.2]{Sh22}, we obtain the following corollary, which provides equivalent characterizations of the little universal Teichm\"uller space (see \cite{Ca,GS}):
\begin{corollary}
Let $\Gamma$ be a quasicircle passing through $\infty$, and let $\Omega$ and $\Omega^*$ denote its complementary components. 
Let $G\colon \mathbb{H}\to\Omega$ and $G^*\colon \mathbb{H}^*\to\Omega^*$ be conformal mappings fixing $\infty$, and let $g\colon \mathbb{R}\to\Gamma$ and $g^*\colon \mathbb{R}\to\Gamma$ be their boundary extensions, respectively.
Define the conformal welding $h\colon \mathbb{R}\to \mathbb{R}$ associated with $\Gamma$ by
\[h\coloneqq g^{-1}\circ g^*.\]
Then the following statements are equivalent:
\begin{enumerate}
  \item $\log G'\in B_0(\mathbb{H})$;
  \item $\Gamma$ is an asymptotically conformal quasicircle relative to $g\colon \mathbb{R}\to\Gamma$;
  \item $g\colon \mathbb{R}\to\mathbb{C}$ is an asymptotically symmetric embedding with image $\Gamma$;
  \item $G$ can be extended to a quasiconformal mapping in the complex plane $\mathbb{C}$ whose complex dilatation $\mu_G(z)= \frac{G_{\bar z}}{G_z}(z)$ satisfies $\frac{G_{\bar z}}{G_z}(x+iy) \to 0$ as $y\to 0_-$ uniformly for $x\in \mathbb{R}$;
  \item $h$ is a symmetric homeomorphism, namely, $h$ is quasisymmetric and
  \[\lim_{t\to 0}\frac{h(x+t)-h(x)}{h(x)-h(x-t)}=1\]
  uniformly for $x\in \mathbb{R}$; and
  \item $|S_{G}(z)|y^2\to 0$ as $y\to 0_+$, where $S_{G}$ is the Schwarzian derivative of $G$.
  \end{enumerate}
  \end{corollary}

\section{Hausdorff dimension of asymptotically conformal curves}\label{sec:hausdorff_dimension_of_asymptotically_conformal_curves}
In this section, we study the geometric size of asymptotically conformal curves from the viewpoint of the Hausdorff dimension. 
It is well known that the Hausdorff dimension of quasicircles, is strictly less than $2$ but can be arbitrarily close to $2$ \cite[Theorem 1]{Smi10}. In contrast, the asymptotically conformal condition forces the Hausdorff dimension to be exactly $1$. In the bounded case, the following result appears as an exercise in \cite[Exercise 11.2.1, p.~250]{Pom}. For completeness, we provide a proof here and extend the conclusion to unbounded asymptotically conformal curves as well.

\begin{theorem}
Let $\Gamma$ be a bounded asymptotically conformal curve. Then the Hausdorff dimension of $\Gamma$ is equal to $1$, that is,
\[
\dim_H(\Gamma) = 1.
\]
\end{theorem}

\begin{proof}
Let $F$ be a conformal mapping from the unit disk $\mathbb{D}$ onto the inner domain $\Omega$ of $\Gamma$. 
Since $\Gamma$ is asymptotically conformal, \cite[Theorem 1]{Po78} guarantees that $\log F' \in B_0(\mathbb{D})$, meaning that 
\[
(1 - |z|^2)\frac{F''(z)}{F'(z)} \to 0 \quad \text{as } |z| \to 1_-.
\]
Hence, for any given $\varepsilon > 0$, there exists some $\delta > 0$ such that
\begin{equation}\label{eq:fpp_est}
\left|\frac{F''(z)}{F'(z)}\right| < \frac{\varepsilon}{1 - |z|^2}, \quad \text{whenever } |z| > 1 - \delta.
\end{equation}


Fix such a constant $\delta$, and consider the closed disk $\overline{D_{1-\delta}}$, where $D_r \coloneqq \{ z \in \mathbb{C} : |z| < r\}$
for $r \in (0,1)$.
Since $F'$ is continuous and $\overline{D_{1-\delta}}$ is compact, there exists a constant $M_\delta$ (depending on $\delta$) such that
\begin{equation}\label{F'_in_}
|F'(z)| \leq M_\delta, \quad \text{for all } z \in \overline{D_{1-\delta}}.
\end{equation}

For any $z \in \mathbb{D} \setminus D_{1-\delta}$, define its radial projection onto the closed disk as $z' \coloneqq \tfrac{1-\delta}{|z|} z$. 
Integrating $F''/F'$ along the segment $[z', z]$ yields
\[
\int_{z'}^{z} \frac{F''(\zeta)}{F'(\zeta)} \, d\zeta = \log F'(z) - \log F'(z').
\]
Taking the real part of both sides and applying \eqref{eq:fpp_est}, we obtain
\begin{align*}
\log |F'(z)| - \log |F'(z')| 
&\leq \int_{|z'|}^{|z|} \left|\frac{F''(\zeta)}{F'(\zeta)}\right| \, dr \leq \int_{0}^{|z|} \frac{\varepsilon}{1 - r} \, dr = \varepsilon \log \frac{1}{1 - |z|}.
\end{align*}
Exponentiating this inequality yields
\[
|F'(z)| \leq \frac{|F'(z')|}{(1 - |z|)^\varepsilon}.
\]
Combining this with \eqref{F'_in_} provides a global estimate:
\begin{equation}\label{eq:fp_growth_d}
|F'(z)| \leq \frac{M_\delta}{(1 - |z|)^\varepsilon}, \quad \text{for all } z \in \mathbb{D}.
\end{equation}

For $z_1, z_2 \in \overline{D_r}$, the segment $[z_1,z_2]$ lies entirely in $\overline{D_r}$. For any $\zeta\in [z_1,z_2]$, we have $1-|\zeta| \geq 1-r$. A direct application of the global bound \eqref{eq:fp_growth_d} gives
\begin{equation}\label{in_D_r}
|F(z_1) - F(z_2)| \leq \int_{[z_1,z_2]}\frac{M_\delta}{(1 - |\zeta|)^\varepsilon}|d\zeta| \leq  M_\delta \frac{|z_1 - z_2|}{(1-r)^{\varepsilon}}, \quad z_1, z_2 \in \overline{D_r}.
\end{equation}

We show that $F$ is globally $\alpha$-H\"older continuous in $\mathbb{D}$ for every $\alpha \in (0,1)$. Fix $\alpha \in (0,1)$ 
and choose $\varepsilon = 1 - \alpha$.

For $z_1, z_2 \in \mathbb{D}$ that lie on a common radial ray with $|z_1| \leq |z_2|$, 
integrating the global bound \eqref{eq:fp_growth_d} along this radial segment $[z_1,z_2]$ by $|d\zeta| = dr$, we have 
\begin{equation}\label{in_ray}
\begin{aligned}
|F(z_1) - F(z_2)| &\leq \int_{|z_1|}^{|z_2|} \frac{M_\delta}{(1 - r)^\varepsilon} \, dr = \frac{M_\delta}{1 - \varepsilon} \left( (1 - |z_1|)^{1-\varepsilon} - (1 - |z_2|)^{1-\varepsilon} \right)\\
&\leq \frac{M_\delta}{1 - \varepsilon} \big(|z_2|-|z_1|\big)^{1-\varepsilon} = \frac{M_\delta}{\alpha}|z_1-z_2|^\alpha.
\end{aligned}
\end{equation}

\medskip

Now consider arbitrary $z_1, z_2 \in \mathbb{D}$. Since $\Gamma$ is a bounded curve, $F$ is bounded; there is some $M>0$
such that $|F(z)| \leq M$ for all $z \in \mathbb{D}$. 
If $|z_1 - z_2| \geq \delta$, the $\alpha$-H\"older condition holds trivially, since
\[
|F(z_1) - F(z_2)| \leq 2M = \frac{2M}{\delta^\alpha} \delta^\alpha \leq \frac{2M}{\delta^\alpha} |z_1 - z_2|^\alpha.
\]
Thus, we may henceforth assume that $|z_1 - z_2| < \delta$.

Define the radius $r_0 \coloneqq 1 - |z_1 - z_2|$. Since $|z_1 - z_2|< \delta$, the relation $r_0 > 1 - \delta$ follows immediately. Define the radial projections of $z_1$ and $z_2$ onto the closed disk $\overline{D_{r_0}}$ as follows:
\[
z_j' \coloneqq 
\begin{cases} 
z_j, & \text{if } |z_j| \leq r_0, \\
\frac{r_0}{|z_j|} z_j, & \text{if } |z_j| > r_0,
\end{cases}
\qquad j=1, 2.
\]
Connect $z_1$ to $z_2$ via a path consisting of at most three segments: the radial segment $[z_1, z_1']$, the straight line segment $[z_1', z_2']$, and the radial segment $[z_2', z_2]$.

First, consider the radial segments. If $z_j \neq z_j'$, the points $z_j$ and $z_j'$ lie on a common ray, and 
\[|z_j - z_j'| = |z_j| - r_0 < 1 - r_0 = |z_1-z_2|.\]
If $z_j = z_j'$, $|F(z_j) - F(z_j')|$ is trivially zero. Applying \eqref{in_ray} yields
\[
|F(z_j) - F(z_j')| \leq \frac{M_\delta}{\alpha} |z_j - z_j'|^\alpha \leq \frac{M_\delta}{\alpha}|z_1-z_2|^\alpha, \quad j=1, 2.
\]

Next, consider the intermediate segment $[z_1', z_2']$, which lies in the closed disk $\overline{D_{r_0}}$. Since the radial projection defined above is a contraction in the distance, 
it follows that $|z_1' - z_2'| \leq |z_1 - z_2|$.
Applying the inequality \eqref{in_D_r} gives
\[
|F(z_1') - F(z_2')| \leq  M_\delta \frac{|z_1' - z_2'|}{(1-r_0)^{\varepsilon}} = M_\delta \frac{|z_1' - z_2'|}{|z_1-z_2|^{\varepsilon}} \leq M_\delta|z_1-z_2|^{1-\varepsilon} = M_\delta|z_1-z_2|^\alpha.
\]

Summing the estimates for these segments via the triangle inequality yields
\[
\begin{aligned}
|F(z_1) - F(z_2)| &\leq |F(z_1) - F(z_1')| + |F(z_1') - F(z_2')| + |F(z_2') - F(z_2)| \\
&\leq \left( \frac{2+\alpha}{\alpha} M_\delta \right) |z_1 - z_2|^\alpha.
\end{aligned}
\]

\medskip
Therefore, $F$ is globally $\alpha$-H\"older continuous in $\mathbb{D}$ for every $\alpha \in (0,1)$.

Finally, by a standard property of the Hausdorff dimension under H\"older mappings (see \cite[Proposition 2.3, p.~32]{Fa}), it follows that
\[
\dim_H \Gamma = \dim_H F(\partial \mathbb{D}) \leq \frac{1}{\alpha} \dim_H (\partial \mathbb{D}) = \frac{1}{\alpha}.
\]
Taking the limit as $\alpha \to 1_-$ gives $\dim_H \Gamma \leq 1$. On the other hand, since $\Gamma$ is a non-degenerate curve, its topological dimension implies $\dim_H \Gamma \geq 1$. Therefore, $\dim_H \Gamma = 1$.
\end{proof}

\begin{corollary}
Let $\Gamma$ be an unbounded asymptotically conformal quasicircle. Then its Hausdorff dimension satisfies
\[
\dim_H(\Gamma) = 1.
\]
\end{corollary}
\begin{proof}
 The result follows from Corollary~\ref{localization} and Proposition~\ref{localize_Gamma:label}, combined with the fact that Hausdorff dimension is a local property.
\end{proof}

\section{Asymptotic smoothness and VMOA}\label{sec:asymptotic_smoothness_and_vmoa}

We extend the following result of Pommerenke \cite[Theorem 2]{Po78} to the non-compact setting:
For a conformal mapping $F$ of the unit disk $\mathbb{D}$ onto a bounded Jordan domain, the boundary curve $\partial F(\mathbb{D})$ is an asymptotically smooth quasicircle if and only if $\log F'$ is $\mathrm{VMOA}$ on $\mathbb{D}$.

We first recall the John--Nirenberg inequality in a local form.

\begin{theorem}[John--Nirenberg inequality]\label{J-N:label}
There exist universal constants $C_1,C_2>0$ such that for every interval $I_0\subset\mathbb{R}$, every $u\in \mathrm{BMO}(I_0)$, every interval $I\subset I_0$, and every $\lambda>0$,
\[
\frac{\left|\left\{x\in I:\ |u(x)-u_I|>\lambda\right\}\right|}{|I|}\leq C_1 \exp\!\left(-\frac{C_2\lambda}{\|u\|_{\mathrm{BMO}(I_0)}}\right).
\]
In particular, if $\|u\|_{\mathrm{BMO}(I_0)}<\varepsilon$, then
\[
\sup_{I\subset I_0}\,\frac1{|I|}\int_I e^{\beta |u-u_I|}\,dx\leq 1+\frac{C_1\beta \varepsilon}{C_2-\beta \varepsilon}
\]
for every $\beta<C_2/\varepsilon$. 
\end{theorem}

\begin{proof}
We verify only the latter statement. For any $I\subset I_0$, write
\[
\begin{aligned}
\frac1{|I|}\int_I e^{\beta |u-u_I|}\,dx&=1+\frac1{|I|}\int_0^\infty\beta e^{\beta\lambda}\left|\left\{x\in I:\ |u-u_I|>\lambda\right\}\right|\,d\lambda\\
&\leq 1+ C_1\beta\int_0^\infty\exp\!\left(\beta\lambda-\tfrac{C_2\lambda}{\|u\|_{\mathrm{BMO}(I_0)}}\right)\,d\lambda.\end{aligned}
\]
Since $\beta<\frac{C_2}{\varepsilon}<\frac{C_2}{\|u\|_{\mathrm{BMO}(I_0)}}$,
the integral converges, and hence
\begin{align}\label{est_beta}
 \frac1{|I|}\int_I e^{\beta |u-u_I|}\,dx\leq 1+\frac{C_1\beta \|u\|_{\mathrm{BMO}(I_0)}}{C_2-\beta \|u\|_{\mathrm{BMO}(I_0)}}\leq 1+\frac{C_1\beta \varepsilon}{C_2-\beta \varepsilon}.
\end{align}
\end{proof}
\begin{corollary}\label{J-N_cor:label}
Let $u\in \mathrm{VMO}(\mathbb{R})$. Then for every $\alpha>0$ and every $0<\varepsilon<\tfrac{C_2}{\alpha}$, there exists $\delta>0$ such that
\[
\sup_{|I|<\delta}\,\frac1{|I|}\int_I e^{\alpha |u(x)-u_I|}\,dx \leq 1+\frac{C_1\alpha \varepsilon}{C_2-\alpha \varepsilon}.
\]
\end{corollary}
\begin{proof}
Fix $\alpha>0$. 
Choose $\varepsilon>0$ sufficiently small such that $\alpha<\tfrac{C_2}{\varepsilon}$. 
Since $u\in \mathrm{VMO}(\mathbb{R})$, there exists $\delta>0$ such that whenever $|I|<\delta$, $\|u\|_{\mathrm{BMO}(I)}<\varepsilon$. 
The conclusion now follows from the inequality~\eqref{est_beta}.
\end{proof}

We are now ready to prove the central result of this section, demonstrating that the $\mathrm{VMOA}$ condition implies asymptotic smoothness in the non-compact setting as well. 
Specifically, we will show that the boundary embedding $g\colon \mathbb{R} \to \Gamma$ satisfies condition 
\eqref{D_1}, thereby establishing its asymptotic smoothness.

\begin{theorem}\label{import_VMOA}
Let $G\colon \mathbb{H}\to \Omega$ be a conformal mapping such that
\[
\log G' \in \mathrm{VMOA}(\mathbb{H}).
\]
Let $g\colon \mathbb{R}\to \partial\Omega=\Gamma$ be the boundary extension of $G$.
Then $g$ is an asymptotically smooth embedding onto $\Gamma$.
\end{theorem}

\begin{proof}
Fix $0<\varepsilon<1$. 
Since $\log G'\in \mathrm{VMOA}(\mathbb{H})$, the boundary function
\[
u(x)=\lim_{y\to0^+}\,\log G'(x+iy), \quad \text{a.e.}
\]
belongs to $\mathrm{VMO}(\mathbb{R})$. 
Moreover, the boundary extension $g$ is locally absolutely continuous with $g'(x)=e^{u(x)}$ for almost every $x\in\mathbb{R}$. 
Therefore, for every bounded interval $I=[a,b]\subset\mathbb{R}$, we have
\[
\ell(\wideparen{g(a)g(b)})
=
\int_I |g'(x)|\,dx
\]
and
\[
g(b)-g(a)=\int_I g'(x)\,dx .
\]

Let $C_1, C_2>0$ be the universal constants from the John--Nirenberg inequality (Theorem~\ref{J-N:label}). We define a universal constant $C=\tfrac{\sqrt{2C_1(1+C_1)}}{C_2}$ and choose 
\[
\eta=\min\left\{\tfrac{C_2}{4}, \tfrac{\varepsilon}{4C}\right\}.
\]
Since $u\in \mathrm{VMO}(\mathbb{R})$, for this specific $\eta>0$, there exists $\delta>0$ such that $\|u\|_{\mathrm{BMO}(I)}<\eta$ whenever $|I|<\delta$.

Fix an interval $I=[a,b]$ with $|I|<\delta$. 
Then $g'(x)=e^{u_I}e^{u(x)-u_I}$. 
Using the inequality $|e^z-1|\leq |z|e^{|z|}$ and the Cauchy--Schwarz inequality, we obtain
\begin{equation}\label{est_e-eI}
 \begin{aligned}
 \frac1{|I|}\int_I |e^{u(x)-u_I}-1|\,dx &\leq\frac1{|I|}\int_I |u(x)-u_I |e^{|u(x)-u_I |}\,dx\\
 &\leq \left(\frac1{|I|}\int_I |u(x)-u_I |^2\,dx\right)^{1/2}\left(\frac1{|I|}\int_I e^{2|u(x)-u_I |}\,dx\right)^{1/2}.
 \end{aligned}
\end{equation}

The John--Nirenberg inequality extends to complex-valued $\mathrm{BMO}$ functions by considering the real and imaginary parts separately. 
By Theorem~\ref{J-N:label}, it follows that
\begin{equation*}
 \begin{aligned}
\frac1{|I|}\int_I |u(x)-u_I|^2\,dx&=\frac{2}{|I|}\int_0^\infty t\,\left|\left\{x\in I:\ |u(x)-u_I|>t\right\}\right|\,dt\\
&\leq 2C_1 \int_0^\infty t\exp\!\left(-\frac{C_2 t}{\|u\|_{\mathrm{BMO}(I)}}\right)\,dt\\
&=2C_1\left(\frac{\|u\|_{\mathrm{BMO}(I)}}{C_2}\right)^2\int_0^\infty s e^{-s}\,ds\leq \frac{2C_1}{C_2^2}\eta^2.
\end{aligned}
\end{equation*}

Hence, we have the $L^2$ estimate
\begin{equation}\label{term_1}
 \left(\frac1{|I|}\int_I |u(x)-u_I|^2\,dx\right)^{1/2}\leq\frac{\sqrt{2C_1}}{C_2}\eta.
\end{equation}
Combining \eqref{est_e-eI}, \eqref{term_1} and Corollary~\ref{J-N_cor:label} yields
\begin{equation*}
 \begin{aligned}
 \frac1{|I|}\int_I |e^{u(x)-u_I}-1|\,dx \leq \frac{\sqrt{2C_1}}{C_2}\eta \left(1+\frac{2C_1\eta}{C_2-2\eta}\right)^{1/2}\leq C\eta \leq \frac{\varepsilon}{4}.
 \end{aligned}
\end{equation*}

Therefore, the arc length is bounded by
\begin{equation}\label{est_length}
\begin{aligned}
 \ell(\wideparen{g(a)g(b)})&=|e^{u_I}|\int_I |1+(e^{u(x)-u_I}-1)|\,dx\leq |e^{u_I}|\left(|I|+\int_I |e^{u(x)-u_I}-1|\,dx\right)\\
 &\leq |e^{u_I}||I|\left(1+\tfrac{\varepsilon}{4}\right).
\end{aligned}
\end{equation}
On the other hand, the chord length is bounded from below by
\begin{equation}\label{est_chord}
 \begin{aligned}
 |g(b)-g(a)|&=|e^{u_I}|\left|\int_I (1+(e^{u(x)-u_I}-1))\,dx\right|\geq |e^{u_I}|\left(|I|-\int_I |e^{u(x)-u_I}-1|\,dx\right)\\
 &\geq |e^{u_I}||I|\left(1-\tfrac{\varepsilon}{4}\right).
 \end{aligned}
\end{equation}

Consequently, by \eqref{est_length} and \eqref{est_chord}, for any $0<\varepsilon<1$, there exists $\delta>0$ such that whenever $|a-b|<\delta$, we have
\[
\frac{\ell(\wideparen{g(a)g(b)})}{|g(a)-g(b)|}\leq \frac{1+\varepsilon/4}{1-\varepsilon/4}<1+\varepsilon.
\]
Therefore, $g$ is asymptotically smooth.
\end{proof}

By Theorem~\ref{import_VMOA} together with the definition of asymptotically smooth chord-arc curves, we immediately obtain the following consequence.

\begin{corollary}\label{VMOA->ASmooth}
Assume that $g^{-1}$ is uniformly continuous.
If $\log G' \in \mathrm{VMOA}(\mathbb{H})$ then $\Gamma$ is an asymptotically smooth chord-arc curve. 
\end{corollary}

The following claim is obvious from the definition.

\begin{proposition}\label{asmooth+unifconti}
Let $\gamma\colon \mathbb{R} \to \Gamma$ be an embedding onto
an asymptotically smooth chord-arc curve $\Gamma$.
If $\gamma$ is uniformly continuous then $\Gamma$ is asymptotically smooth relative to $\gamma$, or equivalently, $\gamma$ is an asymptotically smooth embedding.
\end{proposition}

Conversely, we state the implication of $\mathrm{VMOA}$ from asymptotically smoothness in the following form.
The proof uses the original arguments in \cite{Po78} as localized results and check the uniformity of
relevant constants.

\begin{theorem}\label{ASmooth->VMOA}
If $\Gamma$ is an asymptotically smooth chord-arc curve relative to
$g: \mathbb{R} \to \Gamma$,
then $\log G' \in \mathrm{VMOA}(\mathbb{H})$.
\end{theorem}

\proof
By the standard $\mathrm{BMOA}$ theorem for conformal maps onto chord-arc domains, the assumption implies that $\log G' \in \mathrm{BMOA}(\mathbb{H})$. 
Moreover, $g$ is locally absolutely continuous and $\log g'$ belongs to $\mathrm{BMO}(\mathbb{R})$.

For each closed interval $I \subset \mathbb{R}$ of length $1$, let $Q_I$ be the square domain in $\mathbb{H}$
over $I$.
We take the normalized conformal mapping $S_I: \mathbb{D} \to Q_I$ that extends homeomorphically to their boundaries, and set $G_I=G \circ S_I$. 
The normalization of $S_I$ is given equally to all $I$ so that $S_I(0)$ is the center of $Q_I$ and
$S_I'(0)>0$.
Let $E=S_I^{-1}(I) \subset \partial \mathbb{D}$, and let $g_I$ be the extension of $G_I$ to $E$, which is absolutely continuous. 
Let $\frac{1}{2}E$ be the arc in $\partial \mathbb{D}$ with the same center as $E$ but of a half length.
Since $g_I(E)$ is a subarc of $\Gamma$, it satisfies the condition for the asymptotic smoothness.
Under this condition, we will show that $\log g_I'$ satisfies the condition for $\mathrm{VMO}$ on $\frac{1}{2}E$ uniformly no matter which $I$ is chosen. 

We use the quantitative interpretation of the proof of Pommerenke \cite[Theorem 2]{Pom} for
the implication of asymptotic smoothness. 

The assumption on $\Gamma$ in particular implies that $\Gamma$ is an asymptotically conformal quasicircle.
Then, by Theorem \ref{AS->B_0} when we assume that $g$ is uniformly continuous, or by Theorem \ref{newformilation} when we assume that $\Gamma$ is asymptotically smooth relative to $g$, $\log G'$ belongs to $B_0(\mathbb{H})$.
It follows that we can choose 
a constant $r \in (0,1)$ such that any $e^{i\theta} \in \frac{1}{2}E$ satisfies
\begin{equation}\label{radial}
\left| \log G_I'(\rho e^{i\theta})
-\log G_I'(\rho_*e^{i\theta}) \right|
\leq \frac12 \log\frac{1-\rho_*}{1-\rho}
\end{equation}
whenever $r<\rho_*<\rho<1$. 

We also use the following fact that each $G_I(Q_I)$ is a domain in $\Omega$ with a chord-arc boundary with a uniform chord-arc constant.

\begin{lemma}
Let $G:\mathbb{H} \to \Omega$ is a conformal mapping such that
$\Gamma=\partial \Omega$ is a chord-arc curve. 
Then, for the square $Q_I \subset \mathbb{H}$ over any interval $I \subset \mathbb{R}$ of length $1$, the boundary of the image $G(Q_I)$ is a chord-arc curve with uniform constant $L \geq 1$ independent of $I$.
\end{lemma}

\begin{proof}
We divide $\partial Q_I$ into $I$ and $I'=\partial Q_I \setminus I$.
By the assumption, the image $g(I) \subset \Gamma$ of $I$ under the extension $g$ of $G$ to $\mathbb{R}$ is a chord-arc curve.
On the other hand, $I'$ is a quasi-geodesic in the hyperbolic metric of $\mathbb{H}$, and so is
its image $G(I')$ in $\Omega$.
Since $\Omega$ is a uniform domain, which is equivalent to the condition that $\Omega$ is a quasidisk, $G(I')$ is a uniform arc (see \cite[Lemma 2.6]{V}), which in particular implies that $G(I')$ is a chord-arc curve. 
Alternatively, we may decompose $I'$ further into the vertical sides and the horizontal side, and check that 
the images of the vertical sides are chord-arc curves by \cite[Proposition 1.5]{JK}.
Because both $g(I)$ and $G(I')$ satisfy the Ahlfors--David $1$-regular condition and their union $\partial G(Q_I)$ is a quasicircle, we see that $\partial G(Q_I)$ is a chord-arc curve. 
Since all the constants are determined independently of $I$, the chord-arc constant $L$ for $\partial G(Q_I)$ is also uniform.
\end{proof}

The following estimate is used in \cite[Theorem 2]{Pom}, which is a consequence of this geometric condition that the boundary of the image domain is a chord-arc curve. 

\begin{lemma}
\label{boundary-arc-estimate}
Let $F$ be a conformal mapping of $\mathbb{D}$ onto a bounded Jordan domain $W$
whose boundary is a chord-arc with the constant $L \geq 1$. Let $f$ denote the boundary extension of $F$.
For any $r \in (0,1)$, there exists a constant $C_r>0$ depending only on $r$ such that any
$\theta \in [0,2\pi)$ and $t \in [0,\pi]$ satisfy
\[
\int_{\theta}^{\theta+t} |f'(e^{iu})|\,du
\leq C_rL t\, |F'(\xi_t)|,
\]
where
\[
\xi_t=\left(1-\frac{(1-r)t}{\pi}\right)e^{i\theta}.
\]
\end{lemma}

\begin{proof}
Let $\xi_t = \left(1-\frac{(1-r)t}{\pi}\right)e^{i\theta}$ be the interior reference point. 
Note that its Euclidean distance to the unit circle $\partial \mathbb{D}$ is exactly $1 - |\xi_t| = \frac{(1-r)t}{\pi}$. 
By the Koebe Distortion Theorem for univalent functions, the distance $d_F(\xi_t)=\mathrm{dist}(F(\xi_t),\partial W)$ can be bounded by the derivative at $\xi_t$:
\[
d_F(\xi_t) \leq 2(1-|\xi_t|)|F'(\xi_t)| = \frac{2(1-r)}{\pi} t |F'(\xi_t)|.
\]
We consider two boundary arcs $I_1$ and $I_2$ on $\partial \mathbb{D}$ corresponding to the parameter intervals $[\theta-t, \theta]$ and $[\theta+t, \theta+2t]$, respectively. 
By the boundary distortion estimate for conformal mappings of $\mathbb{D}$ into $\mathbb{C}$ (see \cite[Corollary 4.18]{Pom}), there exist suitable points $e^{i\theta_j}\in I_j$ with $\theta_1 \leq \theta <\theta+t \leq \theta_2$ such that the lengths $\ell(F(S_j))$ of the images of the non-Euclidean segments $S_j$ connecting $\xi_t$ to $e^{i\theta_j}$ ($j=1,2$) satisfy
\[
\ell(F(S_j)) \leq d_F(\xi_t) \exp\left[\frac{K}{\omega(\xi_t, I_j)}\right],
\]
where $K$ is an absolute constant, and $\omega(\xi_t, I_j)$ denotes the harmonic measure of the arc $I_j$ at $\xi_t$.

Since the distance $1-|\xi_t| = \frac{(1-r)t}{\pi}$ is proportional to the lengths of the arcs $I_1$ and $I_2$ up to a factor depending only on $r$, the scale invariance of the harmonic measure implies that $\omega(\xi_t, I_j)$ are bounded from below by a positive constant $c_r$ depending only on $r$. 
Consequently, the exponential term is bounded by a constant $\exp(K/c_r)$.

Observing that $|f(e^{i\theta_j}) - F(\xi_t)| \leq \ell(F(S_j))$. 
Applying the triangle inequality yields
\[
|f(e^{i\theta_2})-f(e^{i\theta_1})| \leq |f(e^{i\theta_2})-F(\xi_t)| + |F(\xi_t)-f(e^{i\theta_1})| \leq 2\exp(K/c_r) d_F(\xi_t).
\]
Substituting the Koebe estimate for $d_F(\xi_t)$ into the inequality above, we obtain
\[
|f(e^{i\theta_2})-f(e^{i\theta_1})| \leq \frac{4(1-r)\exp(K/c_r)}{\pi} t |F'(\xi_t)|\eqqcolon C_r t|F'(\xi_t)|.
\]
The property of the chord-arc curve yields
\[
\int_{\theta_1}^{\theta_2}|f'(e^{iu})|\,du \leq L |f(e^{i\theta_2})-f(e^{i\theta_1})|.
\]
Since
\[\int_{\theta}^{\theta+t} |f'(e^{iu})|\,du \leq \int_{\theta_1}^{\theta_2}|f'(e^{iu})|\,du,
\]
combining these three inequalities gives the assertion.
\end{proof}

\begin{proof}[Proof of Theorem \ref{ASmooth->VMOA} continued]
We follow the argument for the former part of the original proof of \cite[Theorem 2]{Pom}.
Let $\Phi=\log G_I'$ and $\varphi=\log g_I'$.
Fix $\zeta=\rho e^{i\theta}$ $(0<\rho<1)$ for $e^{i\theta} \in \frac{1}{2}E$, and put
\begin{equation} \label{realpartlog}
v(t)=\operatorname{Re}\{\varphi(e^{i(\theta+t)})-\Phi(\zeta)\}=\log\left|\frac{g_I'(e^{i(\theta+t)})}{G_I'(\zeta)}\right|
\end{equation}
for $t \in [0,\pi]$. 
By the Poisson kernel
\[
P_\rho(t)=
\frac1{2\pi}\frac{1-\rho^2}{|e^{it}-\rho|^2},
\]
the integral formula of the boundary value of a harmonic function implies 
\[
\int_{-\pi}^{\pi} v(t)P_\rho(t)\,dt=0.
\]
Then, the Schwarz inequality and
the elementary inequalities $e^x\geq 1+x+x^2/2$ for $x\ge0$ and
$e^x\ge1+x$ for $x<0$ give
\begin{equation}\label{baseineq}
\left(\int_{-\pi}^{\pi}|v(t)|P_\rho(t)\,dt\right)^2
\le
8\left(\int_{-\pi}^{\pi}e^{v(t)}P_\rho(t)\,dt-1\right).
\end{equation}
We will prove that
\[
\lim_{\rho \to 1}\,\int_{-\pi}^{\pi}e^{v(t)}P_\rho(t)\,dt=1
\]
uniformly in $\theta$ and $I$, and then apply \eqref{baseineq} to get the vanishment of the left-hand side as $\rho \to 1$.

We first estimate the part away from $t=0$. 
Let $\alpha=\lambda(1-\rho)$,
where $\lambda>0$ will be chosen later. For $\lambda$ large enough and
$\alpha\leq t\leq\pi$, set
\[
\xi_t=\left(1-\frac{(1-r)t}{\pi}\right)e^{i\theta}.
\]
Then \eqref{radial} gives
\[
|G_I'(\xi_t)|\leq \left(\frac{(1-r)t}{\pi(1-\rho)}\right)^{1/2}|G_I'(\zeta)|, \qquad (\zeta=\rho e^{i\theta}).
\]
Using \eqref{realpartlog} and Lemma \ref{boundary-arc-estimate} combined with this, we obtain
\begin{equation}\label{comparison}
\int_0^t e^{v(s)}\,ds=\frac1{|G_I'(\zeta)|}\int_{\theta}^{\theta+t}|g_I'(e^{iu})|\,du \leq C_r L t\,\frac{|G_I'(\xi_t)|}{|G_I'(\zeta)|}\leq K_\Gamma (1-\rho)^{-1/2} t^{3/2},
\end{equation}where $K_\Gamma$ depends only on $L$ and on the above choice of $r$, hence only on $\Gamma$.

Let the leftmost of \eqref{comparison} be $\phi(t)$ and let the rightmost be $\psi(t)$ for $\alpha \leq t \leq \pi$.
By integration by parts, or equivalently by \cite[Lemma 1]{Po78}, this implies
\begin{align*}
\int_\alpha^\pi e^{v(t)}P_\rho(t)\,dt &=\int_\alpha^\pi \phi'(t)P_\rho(t)\,dt\\
&\leq\int_\alpha^\pi \psi'(t) P_\rho(t)\,dt +\psi(\alpha) P_\rho(\alpha)\\
&=\frac{3K_\Gamma}{2(1-\rho)^{1/2}}\int_\alpha^\pi t^{1/2}P_\rho(t)\,dt+\frac{K_\Gamma \alpha^{3/2}}{(1-\rho)^{1/2}}P_\rho(\alpha).
\end{align*}
By $P_\rho(t) \leq \pi(1-\rho)t^{-2}$ and $\alpha=\lambda(1-\rho)$, we have
\[
\int_{\alpha}^{\pi} e^{v(t)}P_\rho(t)\,dt
\leq 4\pi K_\Gamma\lambda^{-1/2}.
\]
The same estimate holds on $[-\pi,-\alpha]$. 
Therefore, given
$\varepsilon>0$, we may choose $\lambda=\lambda(\varepsilon,\Gamma)$ such that
\begin{equation}\label{tailpart}
\int_{\alpha}^{\pi} e^{v(t)}P_\rho(t)\,dt+\int_{-\pi}^{-\alpha} e^{v(t)}P_\rho(t)\,dt\leq 2\varepsilon.
\end{equation}

Next, we estimate the local part. Here, we use the assumption that $\Gamma$ is asymptotically smooth relative to $g$. 
Then, there is some $t_0>0$ such that $e^{i(\theta+t)} \in E$ and
\[
\ell(\wideparen{g_I(e^{i(\theta+t)}) g_I(e^{i\theta})})\leq (1+\varepsilon)|g_I(e^{i(\theta+t)})-g_I(e^{i\theta})| 
\]
for any $\theta$ with $e^{i\theta} \in \frac{1}{2}E$ and $t \in (0,t_0)$. Consequently,
\[
\int_0^t e^{v(s)}\,ds=\frac1{|G_I'(\zeta)|}\int_{\theta}^{\theta+t}|g_I'(e^{iu})|\,du \leq(1+\varepsilon)\frac{|g_I(e^{i(\theta+t)})-g_I(e^{i\theta})|}{|G_I'(\zeta)|}.
\]

Now take $0\leq t\leq\alpha=\lambda(1-\rho)$. 
By considering the Visser--Ostrowski quotient as in \eqref{ij} for $w_i=g_I(e^{i(\theta+t)})$ and $w_j=g_I(e^{i\theta})$, if $\rho$ is sufficiently close to $1$, with the threshold depending only on $\lambda$, $\varepsilon$, and the asymptotic conformality of $\Gamma$, then
\[
\frac{|g_I(e^{i(\theta+t)})-g_I(e^{i\theta})|}{|G_I'(\zeta)|}\leq(1+\varepsilon)t.
\]
Therefore
\[
\int_0^t e^{v(s)}\,ds\leq(1+\varepsilon)^2t,\qquad 0\leq t\leq\alpha.
\]
A second application of integration by parts gives
\begin{align*}
\int_0^\alpha e^{v(t)}P_\rho(t)\,dt\leq(1+\varepsilon)^2\int_0^\alpha P_\rho(t)\,dt\leq(1+\varepsilon)^2\int_0^\pi P_\rho(t)\,dt\leq\frac12+2\varepsilon .
\end{align*}
The same argument on the negative side gives
\[
\int_{-\alpha}^{0} e^{v(t)}P_\rho(t)\,dt\leq\frac12+2\varepsilon.
\]

Combining these local estimates with \eqref{tailpart}, we obtain
\[
\int_{-\pi}^{\pi}e^{v(t)}P_\rho(t)\,dt\leq 1+6\varepsilon
\]
for all $\rho$ sufficiently close to $1$, uniformly in $\theta$. Hence, \eqref{baseineq} gives
\[
\int_{-\pi}^{\pi}|v(t)|P_\rho(t)\,dt\leq 8\sqrt{\varepsilon}.
\]
Since $P_\rho(t)\geq (2\pi(1-\rho))^{-1}$ for $|t|\le1-\rho$, it follows that
\[
\frac1{2(1-\rho)}\int_{-(1-\rho)}^{1-\rho}\left|\operatorname{Re}\varphi(e^{i(\theta+t)})-\operatorname{Re}\Phi(\rho e^{i\theta})\right|\,dt\leq 8\pi\sqrt{\varepsilon}.
\]
This is equivalent to the $\mathrm{VMO}$ condition for 
$\operatorname{Re}\varphi$ on $\frac{1}{2}E$. 

We apply this result to any interval $I \subset \mathbb{R}$. Since $S_I$ is a $C^1$-diffeomorphism of $\frac{1}{2}E$ onto its image $I_0=S_I(\frac{1}{2}E)$ and $g|_{I}=g_I \circ S_I^{-1}$, we also see that $\operatorname{Re}\log g'|_{I_0}$ satisfies the $\mathrm{VMO}$ condition.
Indeed,
\[
\operatorname{Re}\log g'|_{I_0}=\operatorname{Re}\varphi \circ S_I^{-1}|_{I_0}+\operatorname{Re}\log (S_I^{-1})'|_{I_0},
\]
and clearly $\operatorname{Re}\log (S_I^{-1})'|_{I_0}$ is $\mathrm{VMO}$. Concerning $\operatorname{Re}\varphi \circ S_I^{-1}|_{I_0}$, for any interval $J \subset I_0$, we have
\begin{align*}
&\frac{1}{|J|}\int_J|\operatorname{Re}\varphi \circ S_I^{-1}(x)-(\operatorname{Re}\varphi \circ S_I^{-1})_J|\,dx\\
\leq& \frac{2}{|J|}\int_J|\operatorname{Re}\varphi \circ S_I^{-1}(x)-(\operatorname{Re}\varphi)_{S_I^{-1}(J)}|\,dx\\
\leq& \frac{2M^2}{|S_I^{-1}(J)|}\int_{S_I^{-1}(J)} |\operatorname{Re}\varphi(t)-(\operatorname{Re}\varphi)_{S_I^{-1}(J)}|\,dt,
\end{align*}
where $M \geq 1$ is a constant satisfying $M^{-1} \leq |S_I'(t)| \leq M$ for any $t \in \frac{1}{2}E$.
This implies that $\operatorname{Re}\varphi \circ S_I^{-1}|_{I_0}$ has vanishing mean oscillation.
Since the vanishing condition is uniform independently of $I$, and since $\log g' \in \mathrm{BMO}(\mathbb{R})$, it follows that
$\operatorname{Re} \log g' \in \mathrm{VMO}(\mathbb{R})$.

The conjugate function (the Hilbert transform) of a $\mathrm{VMO}$ function is also $\mathrm{VMO}$ (see \cite[Theorem VI.5.2]{Ga}).
Hence $\log g' \in \mathrm{VMO}(\mathbb{R})$.
Since $\log G' \in \mathrm{BMOA}(\mathbb{H})$, this implies that
$\log G' \in\mathrm{VMOA}(\mathbb{H})$.
\end{proof}

\begin{corollary}
If $\Gamma$ is an asymptotically smooth chord-arc curve and $g$ is uniformly continuous, then $\log G' \in \mathrm{VMOA}(\mathbb{H})$.
\end{corollary}

\begin{proof}
This follows from Theorem \ref{ASmooth->VMOA} and Proposition \ref{asmooth+unifconti}.
\end{proof}

We conclude the paper by presenting the proof of Theorem~\ref{thm:intro-vmoa}, an upper half-plane analogue of Pommerenke's $\mathrm{VMOA}$ theorem.

\begin{proof}[Proof of Theorem~\ref{thm:intro-vmoa}]
The implication $(2)\Rightarrow(1)$ follows from Theorem~\ref{ASmooth->VMOA}, and $(1)\Rightarrow(3)$ is established in Theorem~\ref{import_VMOA}. The remaining implication $(3)\Rightarrow(2)$ (in fact, their equivalence) is immediate.
\end{proof}

Combining Theorem~\ref{thm:intro-vmoa}, \cite[Theorem 2.2]{Sh22}, \cite[Theorem 4.1]{WM-1}, and \cite[Theorem 1.4]{MW}
yields the following corollary, which provides equivalent characterizations of the $\mathrm{VMO}$-Teichm\"uller space (see \cite{Sh22}):

\begin{corollary}
Let $\Gamma$ be a quasicircle passing through $\infty$, and let $\Omega$ and $\Omega^*$ denote its complementary components. 
Let $G\colon \mathbb{H}\to\Omega$ and $G^*\colon \mathbb{H}^*\to\Omega^*$ be conformal mappings fixing $\infty$, and let $g\colon \mathbb{R}\to\Gamma$ and $g^*\colon \mathbb{R}\to\Gamma$ be their boundary extensions, respectively.
Define the conformal welding $h\colon \mathbb{R}\to \mathbb{R}$ associated with $\Gamma$ by
\[h\coloneqq g^{-1}\circ g^*.\]
Then the following statements are equivalent:
\begin{enumerate}
  \item $\log G'\in \mathrm{VMOA}(\mathbb{H})$;
  \item $\Gamma$ is an asymptotically smooth chord-arc curve relative to $g\colon \mathbb{R}\to\Gamma$;
  \item $g\colon \mathbb{R}\to\mathbb{C}$ is an asymptotically smooth embedding whose image $\Gamma$ is a chord-arc curve.
  \item $G$ can be extended to a quasiconformal mapping in the complex plane $\mathbb{C}$ whose complex dilatation $\mu_G(z)$ induces a vanishing Carleson measure $|\mu_G(z)|^2y\, dx\, dy$;
  \item $h$ is a strongly symmetric homeomorphism, namely, $h$ is locally absolutely continuous such that $h'$ belongs to the Muckenhoupt $A_\infty$-weights and $\log h'\in \mathrm{VMO}(\mathbb{R})$; and
  \item $S_G$ induces a vanishing Carleson measure $|S_{G}(z)|^2y^3\, dx\, dy$, where $S_G$ is the Schwarzian derivative of $G$.
  \end{enumerate}
  \end{corollary}


\begin{thebibliography}{99}

\bibitem{Ah} L.V. Ahlfors, {\it Lectures on Quasiconformal Mappings}, 2nd ed., Univ. Lect. Ser. 38, American Mathematical Society, 2006.

\bibitem{AIM} K. Astala, T. Iwaniec and G. Martin, {\it Elliptic Partial Differential Equations and Quasiconformal Mappings in the Plane}, Princeton Mathematical Series, Vol. 48, Princeton University Press, 2009. 


\bibitem{BP} J. Becker and Ch. Pommerenke, {\it \"Uber die quasikonforme Fortsetzung schlichter Funktionen}, Math. Z. 61 (1978), 69--80. 

\bibitem{BA} A. Beurling and L.V. Ahlfors, {\it The boundary correspondence under quasiconformal mappings}, Acta Math. {\bf 96} (1956), 125--142.




\bibitem{BY} A. Brania and S. Yang, {\it Asymptotically symmetric embeddings and symmetric quasicircles}, Proc. Amer. Math. Soc. {\bf 132} (2004), 2671--2678.

\bibitem{Ca} L. Carleson, {\it On mappings, conformal at the boundary}, J. Analyse Math. {\bf 19} (1967), 1--13.




\bibitem{Fa} K. Falconer, {\it Fractal Geometry: Mathematical Foundations and Applications}, John Wiley \& Sons, 2013.







\bibitem{Ga} J.B. Garnett, {\it Bounded Analytic Functions}, Grad. Texts in Math. 236, Springer, 2007.

\bibitem{Gi} D. Girela, {\it Analytic functions of bounded mean oscillation}, Complex Function Spaces (Mekrij\"arvi, 1999), pp. 61--170, Univ. Joensuu Dept. Math. Rep. Ser. No. 4, 2001.

\bibitem{GS} F.P. Gardiner and D. Sullivan, {\it Symmetric structure on a closed curve}, Amer. J. Math. {\bf 114} (1992), 683--736.



\bibitem{JK} D.S. Jerison and C.E. Kenig, {\it Hardy spaces, $A_\infty$, and singular integrals on chord-arc domains},
Math. Scand. 50 (1982), 221--247.




\bibitem{Mac}
P. MacManus, {\it Quasiconformal mappings and Ahlfors--David curves},
Trans. Amer. Math. Soc. {\bf 343} (1994), 853--881.

\bibitem{MW}
K. Matsuzaki and H. Wei, {\it Chordal Loewner chains and Teichm\"uller spaces on the half-plane}, Math. Z. {\bf 312} (2026), 63.

 


\bibitem{Po78} Ch. Pommerenke, {\it On univalent functions, Bloch functions and $\mathrm{VMOA}$}, Math. Ann. {\bf 236} (1978), 199--208. 

\bibitem{Pom} Ch. Pommerenke, {\it Boundary Behaviour of Conformal Maps}, Springer, 1992.

\bibitem{Sar75} D. Sarason, {\it Functions of vanishing mean oscillation}, Trans. Amer. Math. Soc. {\bf 207} (1975), 391--405.

\bibitem{Smi10} S. Smirnov, {\it Dimension of quasicircles}, Acta Math. {\bf 205} (2010), 189--197.


\bibitem{Sh22} Y. Shen, {\it $\mathrm{VMO}$-Teichm\"uller space on the real line}, Ann. Acad. Sci. Fenn. Math. {\bf 47} (2022), 57--82.




\bibitem{Su} T. Sugawa, {\it The universal Teichm\"uller space and related topics},
Proceedings of the international workshop on quasiconformal mappings and their applications, pp. 261--289, Narosa Publishing House, 2007.

\bibitem{Tu}
P. Tukia, {\it Extension of quasisymmetric and Lipschitz embeddings of the real line into the plane},
Ann. Acad. Sci. Fenn. Ser. A I Math. {\bf 6} (1981), 89--94.


\bibitem{TV} P. Tukia and J. V\"ais\"al\"a, {\it Quasisymmetric embeddings of metric spaces}, Ann. Acad. Sci. Fenn. Ser. A I Math. {\bf 5} (1980), 97--114.



\bibitem{V} J. V\"ais\"al\"a, {\it Hyperbolic and uniform domains in Banach spaces}, Ann. Acad. Sci. Fenn. Math. {\bf 30} (2005), 261--302.



\bibitem{WM-1} H. Wei and K. Matsuzaki, {\it Beurling--Ahlfors extension by heat kernel, $A_\infty$-weights for $\mathrm{VMO}$, and vanishing Carleson measures}, Bull. London Math. Soc. {\bf 53} (2021), 723--739.








\end{thebibliography}
\end{document}